\documentclass[oneside, 16pt]{amsart}
\usepackage{mathtools}
\usepackage{soul}
\usepackage{multicol,amsmath,graphics,cancel,soul,color,bbm,amsfonts,dsfont,tikz,mathrsfs,amssymb,bm,bbm}
\usepackage{graphicx}
\usepackage[left=1in, right=1in, top=1.1in,bottom=1.1in]{geometry}
\usepackage{relsize}

\usepackage{enumitem}
\DeclareMathOperator{\EE}{E}
\usepackage{tikz}

\usetikzlibrary{positioning,calc}
\tikzset{
  hasse node/.style={circle, draw, inner sep=1.5pt, minimum size=6mm},
  hasse edge/.style={-},
}

\usetikzlibrary{matrix,patterns,positioning}
\numberwithin{equation}{section}

\definecolor{c0}{HTML}{0CB2AF}
\definecolor{c1}{HTML}{8CC22B}
\definecolor{c2}{HTML}{FAC723}
\definecolor{c3}{HTML}{F29222}
\definecolor{c4}{HTML}{E95E50}
\definecolor{c5}{HTML}{800080}

\newcommand\ringring[1]{%
  {
   \mathop{\kern0pt #1}\limits^{
     \vbox to-1.85ex{
       \kern-2ex 
       \hbox to 0pt{\hss\normalfont\kern.1em \r{}\kern-.45em \r{}\hss}%
       \vss 
     }
   }
  }
}

\newtheorem{thm}{Theorem}[section]

\newtheorem{lemma}[thm]{Lemma}

\newtheorem{remark}[thm]{Remark}
\newtheorem{cor}[thm]{Corollary}
\newtheorem{prop}[thm]{Proposition}

\newtheorem{definition}[thm]{Definition}

\usepackage[colorlinks=true, allcolors=green]{hyperref} 
\hypersetup{colorlinks=true, linkcolor=blue}

\newcounter{dummy} \numberwithin{dummy}{section}

\newtheorem{Example}[dummy]{Example}

\def\1{{\rm l}\hskip -0.21truecm 1}

\begin{document}
\title{Freeness for the $G$-circulant Decomposition of the Partial Transpose of Random matrices}

\author[Octavio Arizmendi]{Octavio Arizmendi$^{\star}$}
\address{$\star$) Centro de Investigaci\'on en Matem\'aticas A.C., Department of Probability and Statistics, M\'exico} 
\email{$\star$)octavius@cimat.mx}

\author[Juli\'an Zazueta]{Juli\'an Zazueta-Obeso$^{\dag}$}
\address{$\dag$) Centro de Investigaci\'on en Matem\'aticas A.C., Department of Probability and Statistics, M\'exico}
\email{$\dag$) julian.zazueta@cimat.mx}

\begin{abstract}
We introduce a left $G$-circulant decomposition for matrices indexed by an arbitrary finite group $G$, extending the diagonal decomposition associated with cyclic groups. We show that, when $A\in M_{|G|}(\mathcal A)$ is free from $M_{|G|}(\mathbb C)$, the components arising from the left $G$-circulant decomposition of $A^t$ form a free family, with the components associated with inverse pairs forming $R$-diagonal pairs. We also describe the distributions of these components in terms of the distribution of $A$. Our results recover the cyclic case and show that different group structures of the same order may lead to different free decompositions of the same matrix.\\

\end{abstract}

\maketitle




\section{Introduction}

We consider a self-adjoint \(dm\times dm\) unitarily invariant random matrix \(A_d\), viewed as a $d\times d$ matrix whose entries $(A_{ij})_{1\leq i,j\leq d}$ are $m\times m$ matrices, that we call blocks. The (right) partial transpose map \(\Gamma:M_{dm}(\mathbb{C})\to M_{dm}(\mathbb{C})\), is obtained by transposing each $m\times m$ block individually. Concretely, it is defined by
\begin{equation*}
A_d^\Gamma:=(id_d\otimes t)(A_d)=\begin{bmatrix}
    A_{1,1}^t&A_{1,2}^t&\cdots &A_{1,d}^t\\
    A_{2,1}^t&A_{2,2}^t&\cdots &A_{2,d}^t\\
    \vdots &\vdots &&\vdots\\
    A_{d,1}^t&A_{2,d}^t&\cdots &A_{d,d}^t
\end{bmatrix}\in M_{d}(\mathbb{C})\otimes M_{m}(\mathbb{C}),
\end{equation*}
where \(t:M_{m}(\mathbb{C})\to M_{m}(\mathbb{C})\) denotes the matrix transposition. We are interested in the asymptotic eigenvalue distribution of \(A_d^\Gamma\) as \(d\to\infty\) and $m$ is fixed. In this regime, Arizmendi, Nechita and Vargas \cite{arizmendi2016asymptotic} proved that if $A_d$ has asymptotic eigenvalue distribution $\mu$, then \(A_d^\Gamma\) has asymptotic eigenvalue distribution given by
\begin{equation}\label{ec-6}
    \mu^\Gamma=\big(D_{1/m}\mu^{\boxplus m(m+1)/2}\big)\boxplus \big(D_{-1/m}\mu^{\boxplus m(m-1)/2}\big).
\end{equation}

We revisit the above result, trying to explain the appearance of free convolutions, by writing the partial transpose as a sum of asymptotically free operators, from the perspective of $G$-circulant decompositions. This was done in \cite{arizmendi2021cyclic} when $G$ is the cylic group of order $m$, but here we show that this is not the unique way of giving such a decomposition, by providing, for each group of order $m$, we provide a decomposition of the partial transpose into between $\lfloor \frac{m}{2}
\rfloor+1$ and $m$ asymptotically free components.

\subsection*{Historical background} 
Before describing our findings, we briefly recall some of the previous developments that led to the formulation of this problem. The study of block-modified random matrices was initiated by Aubrun \cite{aubrun2012partial}, who investigated the partial transpose of large Wishart random matrices. More precisely, let $G_1,\ldots,G_{d_1}$ be independent complex Gaussian random matrices of size $d_2\times p$, that is,
\(
G_i=(g_{j,k}^{(i)})_{j,k},
\)
where the entries $g_{j,k}^{(i)}$ are independent standard complex Gaussian random variables. Define
\begin{equation}\label{ec-1}
    W=\frac{1}{d_1d_2}
    \begin{bmatrix}
        G_1\\
        \rule{1cm}{0.5pt}\\
        \vdots\\
        \rule{1cm}{0.5pt}\\
        G_{d_1}
    \end{bmatrix}
    \begin{bmatrix}
        G_1^{*}&|&\cdots&|&G_{d_1}^{*}
    \end{bmatrix}
    =
    \frac{1}{d_1d_2}(G_iG_j^{*})_{i,j}
    \in M_{d_1d_2}(\mathbb C),
\end{equation}
and consider its partial transpose
\begin{equation*}
    W^{\Gamma}
    :=
    (id\otimes t)(W)
    =
    \frac{1}{d_1d_2}(G_jG_i^{*})_{i,j}
    \in M_{d_1d_2}(\mathbb C),
\end{equation*}
where $t$ denotes the transpose map. Aubrun proved that, as $d_1,d_2,$ and $p$ tend to infinity in such a way that
\(
\frac{p}{d_1d_2}\longrightarrow c>0,
\)
the empirical spectral distribution of $W^{\Gamma}$ converges almost surely to a shifted semicircular distribution. This matrix model is of particular interest in Quantum Information Theory because it arises as the partial transpose of a random induced quantum state. In this context, one is interested in determining whether the partial transpose remains positive, a property known as the Positive Partial Transpose (PPT) criterion, see A. Peres \cite{peres1996separability}. Aubrun answered this question by proving the existence of a strong threshold: asymptotically, the partial transpose is positive if and only if the Wishart parameter $c$ exceeds the critical value $4$.

This line of research has been further developed in several directions. In particular, different asymptotic regimes have been investigated, as well as more general linear maps acting on the blocks of Wishart matrices; see \cite{banica2013asymptotic,banica2012block,jivulescu2014reduction}. 

A particularly important paper for us, in this more general framework is the work of O. Arizmendi, et al. \cite{arizmendi2016asymptotic}, which showed that operator-valued free probability is the appropriate framework for studying block modifications of unitarily invariant large random matrices. This theory was introduced by D. Voiculescu in \cite{voiculescu2006symmetries,voiculescu1995operations}. More precisely, let $X_d$ be a self-adjoint unitarily invariant random matrix of size $dm\times dm$ and \(
\varphi:M_{m}(\mathbb{C})\to M_{n}(\mathbb{C})
\) a linear map, under analytical assumptions, the asymptotic eigenvalue distribution of the block-modified random matrix
\(
X_d^\varphi:=(id_d\otimes\varphi)(X_d)
\)
as $d\to\infty$ is given by $\mu^\varphi=
    \overset{s}{\underset{i=1}{\boxplus}}
    \bigl(D_{\rho_i/n}\mu\bigr)^{\boxplus \rho_i n}$, where $\mu$ denotes the limiting eigenvalue distribution of $X_d$, and $\rho_i$ are the eigenvalues of the Choi matrix associated with the map $\varphi$. In the particular case where $\varphi$ is the partial transpose, the measure $\mu^\varphi$ is given by \eqref{ec-6}.

   On the other hand, in a related work, J. Mingo and M. Popa \cite{mingo2013real} observed that, when $d_1=2$, the limiting operators of $W$ and $W^\Gamma$ in \eqref{ec-1} can be represented as \[
\omega=\frac{1}{d_1}
\begin{pmatrix}
\omega_{1,1} & \omega_{1,2}\\
\omega_{2,1} & \omega_{2,2}
\end{pmatrix},\qquad \omega^t=
\frac{1}{d_1}
\begin{pmatrix}
\omega_{1,1} & \omega_{2,1}\\
\omega_{1,2} & \omega_{2,2}
\end{pmatrix}.
\]
The operator $\omega$ has a Marchenko--Pastur distribution with parameter $c>0$ and is free from $M_{2}(\mathbb{C})$. From this, the authors prove that 
\(
d_1\omega^t=X_0+X_1,
\)
where \( X_0=
\begin{pmatrix}
\omega_{1,1} & \cdot\\
\cdot & \omega_{2,2}
\end{pmatrix}\) and \(X_1=
\begin{pmatrix}
\cdot & \omega_{2,1}\\
\omega_{1,2} & \cdot
\end{pmatrix},
\) are free (In order to make the matrices more visually readable, in all the article we replace all zero entries by dots). This perspective was further developed by Mingo and Arizmendi \cite{arizmendi2021cyclic} where they work with the diagonal decomposition, described in \eqref{ec-2}, and in the case where $d=2$ the authors recover the result from \cite{mingo2013real}. 

Along the same lines  Mingo, Popa, and Szpojankowski \cite{mingo2022partial} proved that, for a fixed number of blocks, the partial transpose of a Haar unitary random matrix can be decomposed along diagonals into a sum of asymptotically free and identically distributed matrices. They further analyzed the interaction between different block decompositions, obtaining asymptotic freeness under suitable assumptions.
\subsection*{Main Results}
Let us know describe our main contribution. We start with a key observation in \cite{arizmendi2016asymptotic}: the partial transpose admits a decomposition in terms of matrix units. More precisely,
\begin{equation*}
A_d^\Gamma=\sum_{i,j=1}^m(I_d\otimes E_{i,j})A_d(I_d\otimes E_{i,j}),
\end{equation*}
where \((E_{i,j})_{1\leq i,j\leq m}\in M_{m}(\mathbb C)\) are the matrix units. The joint non-commutative distribution of the deterministic family
\(
(I_d\otimes E_{i,j})_{1\le i,j\le m},
\)
with respect to
\(
\tau_d:=\frac{1}{dm}\,\mathbb E\circ\mathrm{Tr},
\)
does not depend on \(d\). Since \(A_d\) is unitarily invariant, then Voiculescu's asymptotic freeness (\cite{voiculescu1991limit}) implies that \(A_d\) is asymptotically free from the family \((I_d\otimes E_{i,j})_{1\le i,j\le m}\). 
Then we have the joint convergence 
\[
\bigl(A_d,(I_d\otimes E_{i,j})_{1\le i,j\le m}\bigr)\longrightarrow  \bigl(a,(e_{i,j})_{1\le i,j\le m}\bigr),
\]
where $a$ and $(e_{i,j})_{1\leq i,j\leq m}$ belong to the  non-commutative probability space \((\mathcal A ,\tau)\) containing a copy of \(M_{m}(\mathbb C)\) generated by abstract matrix units \((e_{i,j})_{1\le i,j\le m}\), and \(a\in\mathcal A\) a self-adjoint random variable free from the family \((e_{i,j})_{1\le i,j\le m}\), and having the limiting distribution of \(A_d\). Equivalently, \(a\) is free from the subalgebra \(M_{ m}(\mathbb C)\subseteq\mathcal A\).

Thus, the limiting distribution of \(A_d^\Gamma\) coincides with the distribution of the element
\begin{equation*}
a^t=\sum_{i,j=1}^me_{i,j}ae_{i,j},
\end{equation*}
and, consequently, the analysis of the asymptotic distribution of $A_d^\Gamma$ reduces to the study of the operator $a^t$. 

Our treatment and motivation comes from the work of J. Mingo and the fist author \cite{arizmendi2021cyclic}, where they consider the notion of diagonal decomposition for matrices. Given a matrix $A\in M_{d}(\mathcal{A})$, one defines 
\begin{equation}
    \label{ec-2}
    A=A_0+A_1+\cdots+A_{d-1},
\end{equation} where each $A_k$ consists of the entries of $A$ lying on the $k$-th cyclic diagonal. More precisely, \((A_{k})_{i,j}=A_{i,j}\) if \(j=k+i\) \(\pmod d\) and $0$ otherwise. In this setting, the authors prove that if $A\in M_{d}(\mathcal{A})$ is free from $M_{d}(\mathbb{C})$, then the diagonal decomposition of $A^t$ yields a free family. Moreover, $\{A_i,A_{d-i}\}$ is an $R$-diagonal pair, thus recovering \eqref{ec-6}.  As made explicit in the title of \cite{arizmendi2021cyclic} the cyclic group plays a fundamental role in their treatment.

In the present work, we aim to continue the study of such decompositions by replacing the cyclic group with a general finite group. This is achieved through the introduction of the left \(G\)-circulant decomposition, which provide a natural framework for extending the cyclic diagonal decomposition to arbitrary finite groups, in particular, beyond the commutative setting. This construction is motivated by the work of Bolaños et. al. \cite{bolanos2023g} on $G$-circulant quantum Markov semigroups, which relies on the theory of left $G$-circulant matrices originally introduced by Diaconis \cite{diaconis1981generating}.


The construction of the left $G$-circulant decomposition is based on endowing a matrix with the structure of a left \(G\)-circulant matrix, whose definition is recalled in Section \ref{sec-G-circulant}. To make an intuition of this construction, let us start with a small example: consider the finite group $\mathbb Z_2\times \mathbb Z_2=\{(0,0),(0,1),(1,0),(1,1)\}$ and let $Y\in M_{4}(\mathbb{C})$ be a left $\mathbb Z_2\times \mathbb Z_2$-circulant matrix given by
\begin{equation*}
    Y=aU_{(0,0)}^\ast+bU_{(0,1)}^\ast+cU_{(1,0)}^\ast+dU^\ast_{(1,1)}=\begin{bmatrix}
        a&b&c&d\\
        b&a&d&c\\
        c&d&a&b\\
        d&c&b&a
    \end{bmatrix},
\end{equation*}
where $g\mapsto U_g$ denotes the left regular representation of $\mathbb Z_2\times Z_2$, described explicitly in the Example \ref{ej-reg-rep}. Now, for any $X\in M_{4}(\mathcal{A})$, we can endow $X$ with the same pattern as $Y$, obtaining:
\[
\scalebox{0.88}{$
\begin{aligned}
X=&\begin{bmatrix}
        X_{1,1}&\cdot&\cdot&\cdot\\ 
        \cdot&X_{2,2} &\cdot&\cdot\\ 
        \cdot&\cdot&X_{3,3}&\cdot\\ 
        \cdot&\cdot&\cdot&X_{4,4}
    \end{bmatrix}+\begin{bmatrix}
        \cdot&X_{1,2}&\cdot&\cdot\\ 
        X_{2,1}&\cdot &\cdot&\cdot\\ 
        \cdot&\cdot&\cdot&X_{3,4}\\ 
        \cdot&\cdot&X_{4,3}&\cdot
    \end{bmatrix} +\begin{bmatrix}
        \cdot&\cdot&X_{1,3}&\cdot\\ 
        \cdot&\cdot &\cdot&X_{2,4}\\ 
        X_{3,1}&\cdot&\cdot&\cdot\\ 
        \cdot&X_{4,2}&\cdot&\cdot
    \end{bmatrix}+\begin{bmatrix}
        \cdot&\cdot&\cdot&X_{1,4}\\ 
        \cdot&\cdot &X_{2,3}&\cdot\\ 
        \cdot&X_{3,2}&\cdot&\cdot\\ 
        X_{4,1}&\cdot&\cdot&\cdot
    \end{bmatrix}.\\[0.25cm]
\end{aligned}
$}
\]
The last splitting of $X$ is called the left $\mathbb Z_2\times\mathbb Z_2$-circulant decomposition of $X$. 
Note that the above decomposition may be rewritten as 
\[X=X\odot U_{(0,0)}^\ast+X\odot U_{(0,1)}^\ast+X\odot U_{(1,0)}^\ast+X\odot U_{(1,1)}^\ast.\]
where $\odot$ denotes the Hadamard product.

The preceding example naturally extends to an arbitrary finite group. Let \(G=\{e,g_1,g_2,\ldots,g_{|G|-1}\}\) be a finite group and let \(A=(A_{i,j})_{1\le i,j\le |G|}\in M_{|G|}(\mathcal{A})\). For each \(g\in G\), define the map \(\varphi^{(g)}:  M_{|G|}(\mathcal{A})\to M_{|G|}(\mathcal A)\) by 
\begin{equation}\label{ec-4}
\big(\varphi^{(g)}(A)\big)_{i,j}=
\begin{cases}
A_{i,j}, & \text{if } g_j=gg_i,\\
0, & \text{otherwise}.
\end{cases}
\end{equation}
The map \(\varphi^{(g)}\) selects precisely the entries corresponding to the nonzero positions of the matrix $U_g^\ast$. The left \(G\)-circulant decomposition of \(A\) is then defined by 
\[
A=\sum_{g\in G}\varphi^{(g)}(A)
=:A_e+A_{g_1}+A_{g_2}+\cdots+A_{g_{|G|-1}},
\]
where
\(A_g:=\varphi^{(g)}(A)
\), for all $g\in G$. Hence, the matrices $\{A_g\}_{g\in G}$ partition the entries of $A$ according to the pattern determined by the left regular representation of $G$.

Our main result, Theorem \ref{main-thm}, shows that if \(A\in M_{|G|}(\mathcal{A})\) is free from \(M_{|G|}(\mathbb C)\), and the left \(G\)-circulant decomposition of \(A^t\) is
\[
A^t=Y_e+Y_{g_1}+\cdots+Y_{g_{|G|-1}},
\]
then $Y_{e},\,\{Y_{g},Y_{g^{-1}}\}$ for all $g\in G\backslash \{e\}$, is a free family. Moreover, $\{Y_g,Y_{g^{-1}}\}$ is $R$-diagonal pair for all $g\in G\backslash\{e\}$. In the particular case where we take $G=\mathbb{Z}_d$, we are in the framework of the diagonal decomposition \eqref{ec-2} and reproduce the results from \cite{arizmendi2021cyclic}.

Another relevant result shows that the distributions of the components arising in the left $G$-circulant decomposition of $A^t$ depend on only three distributions. More precisely, if the empirical spectral distribution of $A$ is $\mu$, then the component $Y_e$ has distribution $$\mu_{Y_e}=\big(D_{\frac{1}{|G|}}\mu\big)^{\boxplus |G|}.$$ On the other hand, all components associated with nontrivial involutions share the common distribution $$\mu_{Y_g}=\big(D_{\frac{1}{|G|}}\mu\big)^{\boxplus\frac{|G|}{2}}\boxplus\big(D_{-\frac{1}{|G|}}\mu\big)^{\boxplus\frac{|G|}{2}},$$ while, for each inverse pair $(g,g^{-1})$, the sum $Y_{g}+Y_{g^{-1}}$ has the common distribution $$ \mu_{Y_g+Y_{g^{-1}}}=\big(D_{\frac{1}{|G|}}\mu\big)^{\boxplus|G|}\boxplus\big(D_{-\frac{1}{|G|}}\mu\big)^{\boxplus|G|}.$$

An important feature of our construction is that different groups of the same order may yield different left \(G\)-circulant decompositions of the same matrix. Consequently, one obtains different families of freely independent components depending on the chosen group structure, see Remark \ref{main-remark}.

Apart from the Introduction, the paper is organized into three sections. Section \ref{preliminaries}, in which we present some preliminaries of free probability, regular representations and $G$-circulant matrices. Section \ref{left-G-circ-decomposition}, where we introduce the left $G$-circulant decompositions and present several technical results. Section \ref{section-freeness} contains the main results of the paper. Finally, Section \ref{section-applicacions} we present applications to the Wishart case.

\section{Preliminaries}\label{preliminaries}
\subsection{Free independence} A \textbf{non-commutative probability space} is a pair $(\mathcal{A},\varphi)$ where $\mathcal{A}$ is a unital algebra and $\varphi:\mathcal{A}\to\mathbb{C}$ is a linear functional such that $\varphi(1_\mathcal{A})=1$. If, furthermore,  $\mathcal{A}$ is a $\ast$-algebra, then we call it a  \textbf{$\ast$-probability space}.

For a given non-commutative probability space $(\mathcal{A},\varphi)$, unital subalgebras $(\mathcal{A}_i)_{i\in I}$ of $\mathcal{A}$ are \textbf{free independent} if for all $n\in \mathbb{N}$, $a_1,\dots,a_n\in\mathcal{A}$ such that $a_j\in\mathcal{A}_{i(j)}$ where $i(1),\dots,i(n)\in I$, $i(1)\neq i(2)\neq \cdots \neq i(n)$, then 
\begin{equation*}
    \varphi\big((a_1-\varphi(a_1)1_\mathcal{A})\cdots(a_n-\varphi(a_n)1_\mathcal{A}) \big)=0.
\end{equation*}
A set $\{a_i\in\mathcal{A}\}$ is free if the unital subalgebras $\text{alg}\langle1_\mathcal{A},a_i\rangle$ for all $i\in I$ are a free family.

For $n\geq 1$, a \textbf{partition} of $[n]:=\{1,\dots,n\}$ is a set $\pi=\{V_1,\dots,V_r\}$ of pairwise disjoint non-empty subsets of $[n]$ such that $V_i\cup\cdots \cup V_r=[n]$. The set of all partitions of $[n]$ is denoted by $\mathcal{P}(n)$. A partition $\pi=\{V_1,\dots,V_r\}$ is called \textbf{non-crossing partition} if there are no distinct blocks $V_s$ and $V_t$ with $a,c\in V_s$  and $b,d\in V_t$ such that $a<b<c<d$.

    For a non-commutative probability space $(\mathcal{A},\varphi)$ and $a\in\mathcal{A}$ the $n$-th \textbf{moment} of $a$ is $\varphi(a^n)$. Suppose $a_1,\dots,a_n\in \mathcal{A}$ and $V=\{i_1<\cdots <i_s\}$ is a block of some partition $\pi\in\mathcal{P}(n)$, we denote $(a_1,\dots,a_n)|_V=(a_{i_1},\dots,a_{i_s})$. Let $(f_n)_{n\geq 1}$ be a sequence of multilinear functionals $f_n:\mathcal{A}^n\to\mathbb{C}$. Then for each $\pi\in\mathcal{P}(n)$ 
    \begin{equation*}
        f_\pi[a_1,\dots,a_n]=\prod_{V\in \pi}f_{|V|}\big((a_1,\dots,a_n)|_V\big).
    \end{equation*}
    The \textbf{free cumulants} \cite{speicher1994multiplicative} are the multilinear functionals $\kappa_n : \mathcal{A}^n \to \mathbb{C}$ defined recursively by the moment--cumulant formula
\begin{equation*}
\varphi(a_1 \cdots a_n)
=
\sum_{\pi \in \mathcal{NC}(n)}
\kappa_\pi[a_1, \ldots, a_n],
\end{equation*}
where $\mathcal{NC}(n)$ is the set of non-crossing partitions of $[n]$.
\begin{thm}[\cite{nica2006lectures}, Thm. 11.16]\label{thm-freenes}
     Consider a non-commutative probability space $(\mathcal{A},\varphi)$ and let $(\kappa_n)_{n\in\mathbb{N}}$ be the corresponding free cumulant functionals. Let $(\mathcal{A}_{i})_{i\in I}$ be unital subalgebras of $\mathcal{A}$. Then the following are equivalent:
     \begin{enumerate}
         \item $(\mathcal{A}_i)_{i\in I}$ are free independent.
         \item For all $n\geq 2$ and for all $a_j\in\mathcal{A}_{i(j)}$ ($j=1,\dots,n$) with $i(1),\dots,i(n)\in I$, we have
         \begin{equation*}
             \kappa_n(a_1,\dots,a_n)=0,
         \end{equation*}
         whenever there exists $1\leq l,k\leq n$ with $i(l)\neq i(k)$.
     \end{enumerate}
\end{thm}
\begin{thm}[\cite{nica2006lectures}, Thm. 11.12] Consider a non-commutative probability space $(\mathcal{A},\varphi)$ and let $(\kappa_\pi)_{\pi\in \mathcal{NC}}$ be the corresponding free cumulants. Suppose $n_1,\dots,n_r$ are positive integers, $n=n_1+\cdots+n_r$ and $a_1,\dots,a_{n_1},a_{n_1+1},\dots,a_{n_1+n_2},\dots,a_{n_1+\cdots +n_r}\in\mathcal{A}$. Let $A_1=a_1\cdots a_{n_1}$, $A_2=a_{n_1+1}\cdots a_{n_1+n_2},\ \dots,\  A_r=a_{n_1+\cdots+n_{r-1}+1}\cdots a_{n_1+\cdots +n_r}$. Then
\begin{equation}\label{products-as-arguments}
    \kappa_r(A_1,A_2,\dots,A_r)=\sum_{\substack{\pi\in \mathcal{NC}(n)\\\pi\vee\sigma=1_{n}}}\kappa_\pi[a_1,\dots,a_n],
\end{equation}
    where $\sigma=\big\{\{1,2,\dots,n_1\},\dots,\{n_1+n_1+\cdots+n_{r-1}+1,\dots,n_1+n_1+\cdots+n_r\}\big\}$.
\end{thm}
The following lemma slightly generalizes Theorem \ref{thm-freenes}. Its importance lies in the fact that it allows one to verify freeness between two finite families without passing to the unital algebras they generate. In fact, it suffices to check that all mixed free cumulants vanish when evaluated directly on the variables themselves, yielding a more concrete and tractable criterion in practice.
\begin{lemma}\label{lemma-ind-libre-conjuntos}
    Consider a non-commutative probability space $(\mathcal{A},\varphi)$ and let $(\kappa_n)_{n\in\mathbb{N}}$ be the corresponding free cumulant functionals. Let $(a_i)_{i\in I}\subset \mathcal{A}$ be a family of random variables. Then the following two statements are equivalent:
    \begin{enumerate}
    \item The family $\left\{a_1,\dots,a_m\right\}$ is free from $\left\{a_{m+1}, a_{m+2}, \dots, a_{p}\right\}$.

    \item We have that for all $n\geq2$ and for all $i(1),\cdots,i(n)\in [p]$, we have $$\kappa_n(a_{i(1)},\cdots,a_{i(n)})=0,$$ whenever there exist $1\leq l,k\leq p$ such that $1\leq i(l)\leq m$ and $m+1\leq i(k)\leq p.$
        \end{enumerate}
\end{lemma}
\begin{proof}
Assume that the family $\left\{a_1,\dots,a_m\right\}$ is free from the family 
    $\left\{a_{m+1}, a_{m+2}, \dots, a_{p}\right\}$. This means the unital algebras $\mathcal{A}_1=alg \left<1,a_1,\dots,a_m\right>$ and $\mathcal{A}_2=alg \left<1,a_{m+1},\dots,a_p\right>$ are free. By Theorem \ref{thm-freenes}, all mixed free cumulants vanish.  
Hence, for any $n\geq 2$ and any indices $i(1),\dots,i(n)$ that are not all in the same family, we have
\[
\kappa_n(a_{i(1)},\dots,a_{i(n)})=0,
\]
which proves (2).

Assume that the condition $(2)$ holds. Let $b_j\in\mathcal{A}_{r(j)}$ for $j=1,\dots,n$ with $r(j)\in \{1,2\}$, and suppose that there exist indices
$1\leq l,k\leq n$ such that $r(l)=1$ and $r(k)=2$. Then we have to show that $\kappa_n(b_1,\dots,b_n)$ vanishes. 
Each $b_j$ is a noncommutative polynomial in the generators of its algebra.  
Since free cumulants are multilinear and vanish whenever one of the arguments is $1$ (for $n\geq 2$), it suffices to consider the case where each $b_j$ is a monomial of the form
\begin{equation*}
    b_j=\prod_{i=1}^pa_i^{\beta_i(j)},\qquad\text{for some }\beta_i^{(j)}\in \mathbb{N},
\end{equation*}
where $\beta_i(j)=0$ if $r(j)=2$ for all $i=m+1,\dots,p$, and $\beta_i(j)=0$ if $r(j)=1$ for all $i=1,\dots,m$. Then
\begin{equation*}
    \kappa_n(b_1,\dots,b_l,\dots,b_k,\dots,b_n)=\kappa_n\left(\prod_{i=1}^pa_i^{\beta_i(1)},\dots,\prod_{i=1}^ma_i^{\beta_i(l)},\dots,\prod_{i=m+1}^pa_i^{\beta_i(k)},\dots,\prod_{i=1}^pa_i^{\beta_i(n)}\right).
\end{equation*}
Now by the product as arguments formula \ref{products-as-arguments} we get
\begin{equation*}
    \kappa_n(b_1,\dots,b_l,\dots,b_k,\dots,b_n)=\sum_{\overset{\pi\in \mathcal{NC}(N)}{\pi\vee\sigma=1_{N}}}\kappa_\pi\Big[a_1^{\beta_1(1)},\dots,a_p^{\beta_p(n)}\Big],
\end{equation*}
where $N = \sum_{j=1}^n \sum_{i=1}^p \beta_i(j)$ is the total number of factors, and  $\sigma$ is a interval partition whose blocks correspond to the consecutive factors in each product defining $b_j$. The condition $\pi \vee \sigma = 1_N$ means that all blocks of $\sigma$ are connected by $\pi$.  
Therefore, for a term in the sum to be nonzero, the cumulant $\kappa_\pi\Big[a_1^{\beta_1(1)},\dots,a_p^{\beta_p(n)}\Big]$ must connect variables coming from different $b_j$’s.
But since at least one $b_l$ belongs to $\mathcal{A}_1$ and one $b_k$ belongs to $\mathcal{A}_2$, any such partition $\pi$ necessarily produces at least one block containing both:
\[
a_i \quad (1\leq i\leq m)
\qquad \text{and} \qquad
a_j \quad (m+1\leq j\leq p).
\]
By assumption (2), all such mixed cumulants vanish.  
Hence every term in the above sum is zero, and therefore
\(
\kappa_n(b_1,\dots,b_n)=0.
\)
This shows that all mixed cumulants between $\mathcal{A}_1$ and $\mathcal{A}_2$ vanish, which is equivalent to freeness by \ref{thm-freenes}. Thus (1) holds.
\end{proof}
Other important result in this direction that will be important in the main results is the following: 
\begin{thm}[\cite{nica2006lectures}, Thm. 14.3]\label{thm: as ind bs}
    Let $(\mathcal{A},\varphi)$ be a non-commutative probability space and consider random variables $a_1,\ldots,a_n,b_1,\ldots,b_n\in\mathcal{A}$ such that $\{a_1,\ldots,a_n\}$, $\{b_1,\ldots,b_n\}$ are freely independent. Then we have
    \begin{equation*}
        \varphi(a_1b_1a_2b_2\cdots a_nb_n)=\sum_{\pi\in\mathcal{NC}(n)}\kappa_\pi[a_1,\ldots,a_n]\varphi_{Kr(\pi)}[b_1,\ldots,b_n].
    \end{equation*}
\end{thm}
\subsection{$R$-diagonal operators} Let $(\mathcal{A},\varphi)$ be a $\ast$-probability space. For $a,a^\ast\in\mathcal{A}$ and $\varepsilon\in\{1,-1\}$, we let $a^{(\varepsilon)}=a$ if $\varepsilon=1$ and $a^{(\varepsilon)}=a^\ast$ if $\varepsilon=-1$. We say that $\{a,a^\ast\}$ is an \textbf{$R$-diagonal pair} (see \cite{nica2006lectures}) if for all $\varepsilon_1,\dots,\varepsilon_n\in\{1,-1\}$ we have
\begin{equation*}
    \kappa_n\big(a^{(\varepsilon_1)},\dots,a^{(\varepsilon_n)}\big)=0,
\end{equation*}
unless $n$ is even and $\varepsilon_i+\varepsilon_{i+1}=0$ for all $1\leq i<n$. We denote the non-vanishing cumulants by 
\begin{equation*}
    \alpha_n:=\kappa_{2n}(a,a^\ast,a,a^\ast,\dots,a,a^\ast)\quad \text{and}\quad \beta_n:=\kappa_{2n}(a^\ast,a,a^\ast,a,\dots,a^\ast,a),\qquad \forall \ n\geq 1.
\end{equation*}
The sequences $(\alpha_n)_{n\geq 1}$ and $(\beta_n)_{n\geq 1}$ are called the \textbf{determining sequences} of $a$. An $R$-diagonal pairs is uniquely determined by its determining sequences (see \cite{nica2006lectures}, Cor. 15.7).

In case where the state $\varphi$ is \textbf{tracial}, which will be relevant for us $\alpha_n = \beta_n$, for all $\ n \geq 1$. 
In this setting, $aa^\ast$ and $a^\ast a$ are identically distributed, and their free cumulants can be calculated as follows,
\begin{equation*}
    \kappa_{n}(aa^\ast,\dots,aa^\ast) = \kappa_{n}(a^\ast a,\dots,a^\ast a) = \sum_{\pi\in\mathcal{NC}(n)} \prod_{V \in \pi} \alpha_{|V|}.
\end{equation*}
\subsection{Regular representations}\label{sec-left-reg-rep}
Let $G=\{e,g_1,\dots,g_{|G|-1}\}$ be a finite group and the Hilbert space
\begin{equation*}
    \ell_2(G)=\bigg\{\alpha:G\to\mathbb{C}\quad \bigg|\quad \sum_{g\in G}|\alpha(g)|^2<\infty\bigg\}.
\end{equation*}
equipped with the inner product $\displaystyle \langle \alpha,\beta\rangle=\sum_{g\in G}\overline{\alpha(g)}\beta(g)$.  
The bounded operators on $\mathcal{B}(\ell_2(G))$ admit the basis given by $\left\{\left|e_g\right\rangle\left\langle e_h\right|:g,h\in G\right\}$. The \textbf{left regular representation} maps each $g\in G$ to a unitary operator 
$U_g: \ell_2(G)\to\ell_2(G)$ defined by
\begin{equation*}
    U_g|e_h\rangle=|e_{gh}\rangle,\qquad 
    U_g^\ast|e_h\rangle=|e_{g^{-1}h}\rangle,
\end{equation*}
which can be written as
\begin{equation}\label{ec-7}
    U_g=\sum_{k\in G}|e_k\rangle\langle e_{g^{-1}k}|,\qquad 
    U_g^\ast=\sum_{k\in G}|e_k\rangle\langle e_{gk}|.
\end{equation}
In matrix form, these operators are given by
\begin{equation*}
    (U_g)_{i,j}=\left\{
    \begin{array}{ccc}
        1  & if & g_i=gg_j,\\
        0  &  &\text{otherwise}, 
    \end{array}\right.
    \qquad \text{and} \qquad 
    (U_g^\ast)_{i,j}=\left\{
    \begin{array}{ccc}
       1  &  if & g_j=gg_i,\\
       0 &  & \text{otherwise}.
    \end{array}\right.
\end{equation*}
Similarly, the \textbf{right regular representation} maps each $g\in G$ to a unitary operator 
$V_g:\ell_2(G)\to\ell_2(G)$ defined by
\begin{equation*}
    V_g|e_h\rangle=|e_{hg}\rangle,\qquad 
    V_g^*|e_h\rangle=|e_{hg^{-1}}\rangle,
\end{equation*}
which can be written as
\begin{equation*}
    V_g=\sum_{k\in G}|e_{k}\rangle\langle e_{kg^{-1}}|,\qquad 
    V_g^*=\sum_{k\in G}|e_{k}\rangle\langle e_{kg}|.
\end{equation*}
In matrix form, these operators are given by
\begin{equation*}
    (V_g)_{i,j}=\left\{
    \begin{array}{ccc}
        1  & if & g_i=g_jg,\\
        0  &  &\text{otherwise}, 
    \end{array}\right.
    \qquad \text{and} \qquad 
    (V_g^\ast)_{i,j}=\left\{
    \begin{array}{ccc}
       1  &  if & g_j=g_ig,\\
       0 &  & \text{otherwise}.
    \end{array}\right.
\end{equation*}
These representations satisfy $U_g U_h = U_{gh}$ and $V_g V_h = V_{hg}$, and moreover $U_g V_h = V_h U_g$ for all $g,h \in G$. Other important property which is directly from the above definitions is that $$\sum_{g\in G}U_g^\ast=\sum_{g\in G}V_g^\ast=id_{|G|},$$
where $id_{|G|}$ is the matrix of size $|G|$, with all entries equal to one.
\begin{Example}\label{ej-reg-rep}
    \begin{enumerate}
        \item Let \(G=\mathbb{Z}_2\times \mathbb{Z}_2\), whose elements are enumerated as
\[
e=g_0,\qquad
g_1=(0,1),\qquad
g_2=(1,0),\qquad
g_3=(1,1).
\]
Observe that each operator \(U_{g_i}^*\) is given by \eqref{ec-7}. Computing the matrices corresponding to
\(U_{g_0}^*\), \(U_{g_1}^*\), \(U_{g_2}^*\), and \(U_{g_3}^*\), we obtain the following left regular representations:

\[
\scalebox{0.88}{$
\begin{aligned}
    U_{e}^\ast=\begin{bmatrix}
        1&\cdot&\cdot&\cdot\\ 
        \cdot&1 &\cdot&\cdot\\ 
        \cdot&\cdot&1&\cdot\\ 
        \cdot&\cdot&\cdot&1
    \end{bmatrix},\quad U_{g_1}^\ast=\begin{bmatrix}
        \cdot&1&\cdot&\cdot\\ 
        1&\cdot &\cdot&\cdot\\ 
        \cdot&\cdot&\cdot&1\\ 
        \cdot&\cdot&1&\cdot
    \end{bmatrix},\quad  U_{g_2}^\ast=\begin{bmatrix}
        \cdot&\cdot&1&\cdot\\ 
        \cdot&\cdot &\cdot&1\\ 
        1&\cdot&\cdot&\cdot\\ 
        \cdot&1&\cdot&\cdot
    \end{bmatrix},\quad U_{g_2}^\ast=\begin{bmatrix}
        \cdot&\cdot&\cdot&1\\ 
        \cdot&\cdot &1&\cdot\\ 
        \cdot&1&\cdot&\cdot\\ 
        1&\cdot&\cdot&\cdot
    \end{bmatrix}.
\end{aligned}
$}
\]
        \item Let \(S_3\) denote the symmetric group on the set \(\{1,2,3\}\), and consider the following enumeration of its elements:
\[
g_0=e,\qquad
g_1=(12),\qquad
g_2=(23),\qquad
g_3=(13),\qquad
g_4=(123),\qquad
g_5=(132).
\]
Computing the matrices corresponding to
\(U_{g_0}^*\), \(U_{g_1}^*\), \(U_{g_2}^*\), \(U_{g_3}^*\), \(U_{g_4}^*\), and \(U_{g_5}^*\), we obtain the following left regular representations:
\[
\scalebox{0.8}{$
\begin{aligned}
&U_e^\ast=
\begin{bmatrix}
1 & \cdot & \cdot & \cdot & \cdot & \cdot\\
\cdot & 1 & \cdot & \cdot & \cdot & \cdot\\
\cdot & \cdot & 1 & \cdot & \cdot & \cdot\\
\cdot & \cdot & \cdot & 1 & \cdot & \cdot\\
\cdot & \cdot & \cdot & \cdot & 1 & \cdot\\
\cdot & \cdot & \cdot & \cdot & \cdot & 1
\end{bmatrix},\quad U_{g_1}^\ast=
\begin{bmatrix}
\cdot & 1 & \cdot & \cdot & \cdot & \cdot\\
1 & \cdot & \cdot & \cdot & \cdot & \cdot\\
\cdot & \cdot & \cdot & \cdot & 1 & \cdot\\
\cdot & \cdot & \cdot & \cdot & \cdot & 1\\
\cdot & \cdot & 1 & \cdot & \cdot & \cdot\\
\cdot & \cdot & \cdot & 1 & \cdot & \cdot
\end{bmatrix},\quad U_{g_2}^\ast=
\begin{bmatrix}
\cdot & \cdot & 1 & \cdot & \cdot & \cdot \\
\cdot & \cdot & \cdot & \cdot & \cdot & 1 \\
1 & \cdot & \cdot & \cdot & \cdot & \cdot \\
\cdot & \cdot & \cdot & \cdot & 1 & \cdot \\
\cdot & \cdot & \cdot & 1 & \cdot & \cdot \\
\cdot & 1 & \cdot & \cdot & \cdot & \cdot
\end{bmatrix},
\\
&U_{g_3}^\ast=
\begin{bmatrix}
\cdot & \cdot & \cdot & 1 & \cdot & \cdot \\
\cdot & \cdot & \cdot & \cdot & 1 & \cdot \\
\cdot & \cdot & \cdot & \cdot & \cdot & 1 \\
1 & \cdot & \cdot & \cdot & \cdot & \cdot \\
\cdot & 1 & \cdot & \cdot & \cdot & \cdot \\
\cdot & \cdot & 1 & \cdot & \cdot & \cdot
\end{bmatrix},\quad U_{g_4}^\ast=
\begin{bmatrix}
\cdot & \cdot & \cdot & \cdot & \cdot & 1 \\
\cdot & \cdot & 1 & \cdot & \cdot & \cdot \\
\cdot & \cdot & \cdot & 1 & \cdot & \cdot \\
\cdot & 1 & \cdot & \cdot & \cdot & \cdot \\
1 & \cdot & \cdot & \cdot & \cdot & \cdot \\
\cdot & \cdot & \cdot & \cdot & 1 & \cdot
\end{bmatrix}
,\quad U_{g_5}^\ast=
\begin{bmatrix}
\cdot & \cdot & \cdot & \cdot & 1 & \cdot \\
\cdot & \cdot & \cdot & 1 & \cdot & \cdot \\
\cdot & 1 & \cdot & \cdot & \cdot & \cdot \\
\cdot & \cdot & 1 & \cdot & \cdot & \cdot \\
\cdot & \cdot & \cdot & \cdot & \cdot & 1 \\
1 & \cdot & \cdot & \cdot & \cdot & \cdot
\end{bmatrix}.
\end{aligned}
$}
\]
    \end{enumerate}
    Observe that the explicit matrix realizations above depend on the chosen enumeration of the elements the group). Different labelings may produce different matrices. However, any two such realizations of the left regular representation are unitarily equivalent, the unitary equivalence being implemented by conjugation with a permutation matrix.
\end{Example}
\subsection{$G$-circulant matrices}
\label{sec-G-circulant}
In this subsection, we introduce the notions of left and right \(G\)-circulant matrices associated with a finite group \(G\). For further details and additional background on \(G\)-circulant matrices, we refer the reader to \cite{diaconis1981generating}.

\begin{definition}\label{def-G-circulantes}
    For a finite group $G$, we say that a square matrix $B\in M_{|G|}(\mathbb{C})$ 
    \begin{enumerate}
        \item is \textbf{right $G$-circulant} if it can be written as
\begin{equation*}
    B=\sum_{g\in G} b_g V_g^\ast,
\end{equation*}
where $g\mapsto V_g$ is the right regular representation of the group $G$.
        \item is \textbf{left $G$-circulant} if it can be written as
\begin{equation*}
    B=\sum_{g\in G} b_g U_g^\ast,
\end{equation*}
where $g\mapsto U_g$ is the left regular representation of the group $G$. 
    \end{enumerate}
\end{definition}
One characterization of the $G$-circulant matrices is that a matrix $B\in M_{|G|}(\mathbb{C})$ is left (or right) $G$-circulant matrices if and only if it commutes with $V_g^\ast$ (or $U_g^\ast$), for all $g\in G$. With this characterization we can produce $G$-circulant matrices as is described in the following lemma.
\begin{lemma}
Let $G$ be a finite group and $A\in M_{|G|}(\mathbb{C})$. Then the matrix $\displaystyle\sum_{g\in G}U_gAU_g^{-1}$ is right $G$-circulant and $\displaystyle\sum_{g\in G}V_gAV_g^{-1}$ is left $G$-circulant.
\begin{proof}
    Let $h\in G$, Then 
\begin{equation*}
    U_h^*\left[\sum_{g\in G}U_g^*A(U_g^*)^{-1}\right](U_h^*)^{-1}=\sum_{g\in G}U_h^*U_g^*A(U_h^*U_g^*)^{-1}=\sum_{g\in G}U_g^*A(U_g^*)^{-1}.
\end{equation*}
 This proves that the matrix $\displaystyle\sum_{g\in G}U_gAU_g^{-1}$ commutes with $U_g^*$, which means that it is right $G-$circulant. The proof of $\displaystyle\sum_{g\in G}V_gAV_g^{-1}$ is left $G$-circulant is analogous.  
\end{proof}
\end{lemma}


\section{Left $G$-circulant decomposition}\label{left-G-circ-decomposition}
Let $G = \{e, g_1, \dots, g_{|G|-1}\}$ be a finite group. We now describe the construction of the left $G$-circulant decomposition of matrices. Consider a matrix 
\(
A = (A_{ij})_{1 \le i,j \le |G|} \in M_{|G| }(\mathcal{A}),
\) 
and observe that
\[
A=id_{|G|}\odot A=\Big(\sum_{g\in G}U_g^\ast \Big)\odot A=\sum_{g\in G}U_g^\ast\odot A,\]
where $id_{|G|}$ is the matrix with all ones of size $|G|$ and $\odot $ represent the Hadamard product. Then for each \(g\in G\), consider the map \(\varphi^{(g)}:M_{|G|}(\mathcal A)\to M_{|G|}(\mathcal A)\) defined in \eqref{ec-4} and it coincides with $\varphi^{(g)}(A)=U_g^\ast\odot A$. Therefore the \textbf{left \(G\)-circulant decomposition} of \(A\) is
\begin{equation*}
    A  =\sum_{g \in G} U_g^* \odot A=\sum_{g\in G}\varphi^{(g)}(A)
 := A_{e} + A_{g_1} + \cdots + A_{g_{|G|-1}},
\end{equation*}
where
\(A_g:=\varphi^{(g)}(A)=U_g^\ast\odot A
\), for all $g\in G$ (to describe the right $G$-circulant case, we just change $U_g$ by $V_g$).
\begin{Example}
Building upon Example~\ref{ej-reg-rep}, where the left regular representations of \(\mathbb{Z}_2\times\mathbb{Z}_2\) and \(S_3\) were computed, we now derive the corresponding left \(G\)-circulant decompositions.
\begin{enumerate}
    \item Consider the finite abelian group
\(\mathbb{Z}_2\times\mathbb{Z}_2
=\{e,(0,1),(1,0),(1,1)\},
\)
and let \(A\in M_{4}(\mathcal{A})\). The left
\(\mathbb{Z}_2\times\mathbb{Z}_2\)-circulant decomposition of \(A\) is given by
\[
A=\sum_{g\in\mathbb{Z}_2\times\mathbb{Z}_2}U_g^*\odot A
=A_{e}+A_{g_1}+A_{g_2}+A_{g_3},
\]
This means that
    \[
\scalebox{0.88}{$
\begin{aligned}
A=&\begin{bmatrix}
        A_{1,1}&\cdot&\cdot&\cdot\\ 
        \cdot&A_{2,2} &\cdot&\cdot\\ 
        \cdot&\cdot&A_{3,3}&\cdot\\ 
        \cdot&\cdot&\cdot&A_{4,4}
    \end{bmatrix}+\begin{bmatrix}
        \cdot&A_{1,2}&\cdot&\cdot\\ 
        A_{2,1}&\cdot &\cdot&\cdot\\ 
        \cdot&\cdot&\cdot&A_{3,4}\\ 
        \cdot&\cdot&A_{4,3}&\cdot
    \end{bmatrix} +\begin{bmatrix}
        \cdot&\cdot&A_{1,3}&\cdot\\ 
        \cdot&\cdot &\cdot&A_{2,4}\\ 
        A_{3,1}&\cdot&\cdot&\cdot\\ 
        \cdot&A_{4,2}&\cdot&\cdot
    \end{bmatrix}+\begin{bmatrix}
        \cdot&\cdot&\cdot&A_{1,4}\\ 
        \cdot&\cdot &A_{2,3}&\cdot\\ 
        \cdot&A_{3,2}&\cdot&\cdot\\ 
        A_{4,1}&\cdot&\cdot&\cdot
    \end{bmatrix}\\[0.25cm]
    =&A_{e}+A_{(0,1)}+A_{(1,0)}+A_{(1,1)}.
\end{aligned}
$}
\]
    \item Consider the smallest non-commutative group, the symmetric group on $\{1,2,3\}$, whose elements are
   $ S_3=\{e,(12),(23),(13),(123),(132)\}$, and a matrix $A \in M_{6}(\mathcal{A})$. The left $S_3$-circulant decomposition of the matrix $A$ is given by
    \begin{equation*}
        A=\sum_{g\in S_3}U_g^*\odot A=A_{e}+A_{g_1}+\cdots +A_{g_5},
    \end{equation*}
This means that

\[
\scalebox{0.80}{$
\begin{aligned}
A =&
\begin{bmatrix}
A_{1,1} & \cdot & \cdot & \cdot & \cdot & \cdot\\
\cdot & A_{2,2} & \cdot & \cdot & \cdot & \cdot\\
\cdot & \cdot & A_{3,3} & \cdot & \cdot & \cdot\\
\cdot & \cdot & \cdot & A_{4,4} & \cdot & \cdot\\
\cdot & \cdot & \cdot & \cdot & A_{5,5} & \cdot\\
\cdot & \cdot & \cdot & \cdot & \cdot & A_{6,6}
\end{bmatrix}
+
\begin{bmatrix}
\cdot & A_{1,2} & \cdot & \cdot & \cdot & \cdot\\
A_{2,1} & \cdot & \cdot & \cdot & \cdot & \cdot\\
\cdot & \cdot & \cdot & \cdot & A_{3,5} & \cdot\\
\cdot & \cdot & \cdot & \cdot & \cdot & A_{4,6}\\
\cdot & \cdot & A_{5,3} & \cdot & \cdot & \cdot\\
\cdot & \cdot & \cdot & A_{6,4} & \cdot & \cdot
\end{bmatrix}+
\begin{bmatrix}
\cdot & \cdot & A_{1,3} & \cdot & \cdot & \cdot \\
\cdot & \cdot & \cdot & \cdot & \cdot & A_{2,6} \\
A_{3,1} & \cdot & \cdot & \cdot & \cdot & \cdot \\
\cdot & \cdot & \cdot & \cdot & A_{4,5} & \cdot \\
\cdot & \cdot & \cdot & A_{5,4} & \cdot & \cdot \\
\cdot & A_{6,2} & \cdot & \cdot & \cdot & \cdot
\end{bmatrix}
\\
&+
\begin{bmatrix}
\cdot & \cdot & \cdot & A_{1,4} & \cdot & \cdot \\
\cdot & \cdot & \cdot & \cdot & A_{2,5} & \cdot \\
\cdot & \cdot & \cdot & \cdot & \cdot & A_{3,6} \\
A_{4,1} & \cdot & \cdot & \cdot & \cdot & \cdot \\
\cdot & A_{5,2} & \cdot & \cdot & \cdot & \cdot \\
\cdot & \cdot & A_{6,3} & \cdot & \cdot & \cdot
\end{bmatrix}+
\begin{bmatrix}
\cdot & \cdot & \cdot & \cdot & \cdot & A_{1,6} \\
\cdot & \cdot & A_{2,3} & \cdot & \cdot & \cdot \\
\cdot & \cdot & \cdot & A_{3,4} & \cdot & \cdot \\
\cdot & A_{4,2} & \cdot & \cdot & \cdot & \cdot \\
A_{5,1} & \cdot & \cdot & \cdot & \cdot & \cdot \\
\cdot & \cdot & \cdot & \cdot & A_{6,5} & \cdot
\end{bmatrix}
+
\begin{bmatrix}
\cdot & \cdot & \cdot & \cdot & A_{1,5} & \cdot \\
\cdot & \cdot & \cdot & A_{2,4} & \cdot & \cdot \\
\cdot & A_{3,2} & \cdot & \cdot & \cdot & \cdot \\
\cdot & \cdot & A_{4,3} & \cdot & \cdot & \cdot \\
\cdot & \cdot & \cdot & \cdot & \cdot & A_{5,6} \\
A_{6,1} & \cdot & \cdot & \cdot & \cdot & \cdot
\end{bmatrix}\\[0.25cm]
=&A_{e}+A_{(12)}+A_{(23)}+A_{(13)}+A_{(123)}+A_{(132)}.
\end{aligned}
$}
\]
\end{enumerate}
 
\end{Example}
Let us denote by $\mathcal{D}(\mathbb{C})\subset \mathcal{D}(\mathcal{A})\subset M_{|G|}(\mathcal{A})$, the subalgebras of scalar diagonal matrices and diagonal matrices with entries in $\mathcal{A}$, respectively. We denote by $\EE:M_{|G|}(\mathcal{A})\longrightarrow \mathcal{D}(\mathcal{A})$ the conditional expectation from $M_{|G|}(\mathcal{A})$ to $\mathcal{D}(\mathcal{A})$. For $A=(A_{ij})_{i,j}\in M_{|G|}(\mathcal{A})$, it is given by
\begin{equation*}
    \EE(A)=\begin{bmatrix}
        A_{1,1}&\cdot&\cdots & \cdot\\
        \cdot&A_{2,2}&\cdots &\cdot\\ 
        \vdots &\vdots &\ddots&\vdots \\
        \cdot&\cdot&\cdots &A_{|G|,|G|}
    \end{bmatrix}=\sum_{h\in G} \big\langle e_h ,A e_{h} \big\rangle  \ \  |e_h\rangle \langle e_{h}|=\sum_{h\in G}E_{h,h}AE_{h,h},
\end{equation*}
where $E_{h,h}=|e_h\rangle\langle e_e|$. Note that for any finite group $G$ and $A\in M_{|G|}(\mathcal{A})$, we have that $\EE(A)=A_e$.

\begin{lemma}\label{l1} Let $X \in M_{|G|}(\mathcal{A})$. Then for every $k=0,1,\ldots,|G|-1$ the following holds:
\begin{enumerate}
\item 
    \begin{equation*}
        \big(XU_{g_k}^*\big)_{ij} = x_{i,r}, \qquad \text{if } \; g_r = g_k ^{-1}g_j,
    \end{equation*}
    \item 
    \begin{equation*}
        \big(U_{g_k}^*X\big)_{ij} = x_{s,j}, \qquad \text{if } \; g_s = g_k g_i,
    \end{equation*}
    \item 
    \begin{equation*}
        \EE\big(XU_{g_k}^*\big) = \EE\big((U_{g_k}^*)^{-1} X^{t}\big).
    \end{equation*}
\end{enumerate}
    
\end{lemma}
\begin{proof}
    To prove the item (1), we observe that \begin{equation*}
        \big(XU_{g_k}^*\big)_{i,j}=\sum_l(X)_{i,l}(U_{g_k}^*)_{l,j} =\sum_lx_{i,l}\delta_{g_j,g_kg_l}= x_{i,r}, \qquad \text{if } \; g_r = g_k ^{-1}g_j.
    \end{equation*}
Similarly, for item (2), we have
    \begin{equation*}
        \big(U_{g_k}^*X\big)_{i,j}=\sum_{l}(U_{g_k}^*)_{i,l}(X)_{l,j} =\sum_l\delta_{g_l,g_kg_i}x_{l,j} = x_{s,j}, \qquad \text{if } \; g_s = g_k g_i.
    \end{equation*}
    Finally, note that $\big(U_{g_k}^*\big)^{-1}=U_{g_k}$, which implies that
    \begin{equation}\label{item3}
        \Big(\big(U_{g_k}^*\big)^{-1}(X^t)\Big)_{i,i}=\sum_l\big(U_{g_k}\big)_{i,l}(X^t)_{l,i}=\sum_l\delta_{g_i,g_kg_l}x_{i,l}=x_{i,r},\qquad \text{if } \;  g_r=g_k^{-1}g_i.
    \end{equation}
Now, by the first item we have that $\big(XU_{g_k}^*\big)_{ii}=x_{i,r}$ if $g_r=g_k^{-1}g_i$ which is exactly the equation \eqref{item3}. This implies that  \begin{equation*}
        \EE\big(XU_{g_k}^*\big) = \EE\big((U_{g_k}^*)^{-1} X^{t}\big).
    \end{equation*}
\end{proof}
\begin{lemma}\label{lemmaX_k}
 Let $X \in M_{|G|}(\mathbb{\mathcal{A}})$. If $X_{e}, X_{g_1}, \dots, X_{g_{|G|-1}}$ is the left $G$-circulant decomposition of $X$.
 Then for every $k=0,1,\ldots,|G|-1$
\[
    X_{g_k} \;=\; \EE\!\big(X(U_{g_k}^*)^{-1}\big)\,U_{g_k}^* \;=\; U_{g_k}^*\,\EE\!\big((U_{g_k}^*)^{-1}X\big).
\]
\end{lemma}
\begin{proof}
    Lets prove that $ X_{g_k} \;=\; \EE\!\big(X(U_{g_k}^*)^{-1}\big)\,U_{g_k}^* $. First note that 
    \begin{equation*}
        \Big(\EE\!\big(X(U_{g_k}^*)^{-1}\big)\,U_{g_k}^*\Big)_{i,j}=\sum_l\Big(\EE\big(X(U_{g_k})^{-1}\big)\Big)_{i,l}\big(U_{g_k}^*\big)_{l,j}=\Big(X\big(U_{g_k}^*\big)^{-1}\Big)_{i,i }\big(U_{g_k}^*\big)_{i,j}.
    \end{equation*}
    Now 
    \begin{equation*}
        \Big(X\big(U_{g_k}^*\big)^{-1}\Big)_{i,i }=\sum_l(X)_{i,l}\big(U_{g_k}\big)_{l,i}=x_{i,s},\qquad\text{if }\; g_s=g_kg_i. 
    \end{equation*}
    Then
    \begin{equation*}
        \Big(\EE\!\big(X(U_{g_k}^*)^{-1}\big)\,U_{g_k}^*\Big)_{i,j}=x_{i,s}\big(U_{g_k}^*\big)_{i,j},\qquad\text{if }\; g_s=g_kg_i. 
    \end{equation*}
    Since  $\big(U_{g_k}^*\big)_{i,j}=1$ if $g_j=g_kg_i$, then $s=j$ and it implies that $\Big(\EE\!\big(X(U_{g_k}^*)^{-1}\big)\,U_{g_k}^*\Big)_{i,j}=x_{i,j}.$ So in all cases
    \begin{equation*}
        \Big(\EE\!\big(X(U_{g_k}^*)^{-1}\big)\,U_{g_k}^*\Big)_{i,j}=\big(U_{g_k}^*\big)_{i,j}x_{i,j}=\big(U_{g_k}^*\odot X\big)_{i,j}=\big(X_{g_k}\big)_{i,j},
    \end{equation*}
    which implies that $X_{g_k} \;=\; \EE\!\big(X(U_{g_k}^*)^{-1}\big)\,U_{g_k}^*$. To prove the other equivalence, $X_{g_k} \;=\;U_{g_k}^*\,\EE\!\big((U_{g_k}^*)^{-1}X\big)$, we just need to do an analogous computations as above. 
\end{proof}

\begin{lemma}\label{lemma Y-UXU}
    Let $X\in M_{|G|}(\mathcal{A})$.
    \begin{enumerate}
        \item If $X_{e}, X_{g_1}, \dots, X_{g_{|G|-1}}$ is the left $G$-circulant decomposition of $X$, then
        \begin{equation*}
            X_{g_k} = \sum_{h \in G} E_{h,h}\, X\, E_{g_kh,g_kh},\qquad \forall \ k=0,1,\ldots,|G|-1.
        \end{equation*}
        \item If $Y_{e}, Y_{g_1}, \dots, Y_{g_{|G|-1}}$ is the left $G$-circulant decomposition of $X^t$, then
        \begin{equation*}
            Y_{g_k} = U_{g_k}^*\left[\sum_{h \in G} E_{h,h}\, X\, E_{g_k^{-1}h,\,g_k^{-1}h}\right] U_{g_k}^*=U_{g_k}^*X_{g_k^{-1}}U_{g_k}^*,\qquad \forall \ k=0,1,\ldots,|G|-1.
        \end{equation*}
    \end{enumerate}
\end{lemma}
\begin{proof}
    For the first point, recall that for the Lemma \ref{lemmaX_k} $X_{g_k}=U_{g_k}^\ast\EE\big((U_{g_k}^\ast)^{-1}X\big)$, it implies that
    \begin{equation*}
        X_{g_k}
=
U_{g_k}^\ast\left[\sum_{h\in G}E_{h,h}(U_{g_k}^*)^{-1}XE_{h,h}\right]=\sum_{h\in G}U_{g_k}^\ast E_{h,h}(U_{g_k}^*)^{-1}XE_{h,h}=\sum_{h\in G}E_{g_k^{-1}h,g_k^{-1}h}XE_{h,h}.
    \end{equation*}
    Renaming the index \(h\mapsto g_k^{-1}h\), we get
\[
X_{g_k}
=
\sum_{h\in G}E_{h,h}XE_{g_kh,g_kh}.
\]
   To prove the second point. We know from Lemma \ref{l1} that $\EE\big(XU_{g_k}^*\big)=\EE\big((U_{g_k}^*)^{-1}X^t\big)$ and from Lemma \ref{lemmaX_k} $Y_{g_k}=U_{g_k}^\ast\EE((U_{g_k}^\ast)^{-1}X^t)$, it implies that
    \begin{align*}
        Y_{g_k}=U_{g_k}^*\EE(XU_{g_k}^*)=U_{g_k}^*\left[\sum_{h\in G}E_{h,h}XU_{g_k}^*E_{h,h}\right]=&U_{g_k}^*\left[\sum_{h\in G}E_{h,h}XU_{g_k}^*E_{h,h}\big(U_{g_k}^*\big)^{-1}\right]U_{g_k}^*\\=&U_{g_k}^*\left[\sum_{h\in G}E_{h,h}XE_{g_k^{-1}h,g_k^{-1}h}\right]U_{g_k}^*.
    \end{align*}
\end{proof}
\section{Freeness for left $G$-circulant decomposition}\label{section-freeness}

In this section, we prove the main theorem of the paper. More precisely, under the assumption that $X\in M_{|G|}(\mathcal{A})$ is free from $M_{|G|}(\mathbb{C})$, we show that the components arising from the left $G$-circulant decomposition of $X$ form a free family over $\mathbb{C}$. As a consequence, we recover the results of Arizmendi, Nechita, and Vargas in \eqref{ec-6}, which we obtain by identifying the explicit distributions of the individual components in the left $G$-circulant decomposition.



%


\subsection{Cumulants for left $G$-circulant decomposition}

We first establish a few lemmas concerning free cumulants for the components arising from the left $G$-circulant decomposition of an arbitrary matrix $X\in M_{|G|}(\mathcal{A})$. We emphasize that these results do not require any freeness assumption between $X$ and the deterministic matrices $M_{|G|}(\mathbb{C})$.

\begin{lemma}\label{l2}
    Let $X_{e}, X_{g_1}, \dots, X_{g_{|G|-1}}$ be the left $G$-circulant decomposition of 
    $X \in M_{|G|}(\mathcal{A})$. Then
    \begin{equation*}
        \kappa_{r}(X_{g_{i_1}}, X_{g_{i_2}}, \dots, X_{g_{i_r}}) = 0,
    \end{equation*}
    whenever $g_{i_1} g_{i_2} \cdots g_{i_r} \neq e$.
\end{lemma}
\begin{proof}
    Using the moment–cumulant formula, we have
    \begin{equation*}
        \kappa_r(X_{g_{i_1}}, X_{g_{i_2}}, \dots, X_{g_{i_r}})
        = \sum_{\pi \in \mathcal{NC}(r)} \mu(0_r,\pi)\,
        \varphi_\pi\big[X_{g_{i_1}}, X_{g_{i_2}}, \dots, X_{g_{i_r}}\big],
    \end{equation*}
    where for a block $V = \{b_1 < \dots < b_s\}$,
    \[
       \varphi_\pi[X_{g_{i_1}},\dots,X_{g_{i_r}}]
        = \prod_{V\in\pi} 
        \varphi\!\left(X_{g_{i_{b_1}}}\cdots X_{g_{i_{b_s}}}\right).
    \]
    Suppose that $g_{i_1} g_{i_2} \cdots g_{i_r} \neq e$.  
    Then, for any partition $\pi \in \mathcal{NC}(r)$, there exists at least one block 
    \(V = \{b_1,\dots,b_s\}\) such that
    \[
        g_{i_{b_1}} g_{i_{b_2}} \cdots g_{i_{b_s}} \neq e,
        \qquad \text{(this was proven by induction on $r$)}.
    \]
    We now prove that 
    \begin{equation*}
        \varphi\!\left(X_{g_{i_{b_1}}}\cdots X_{g_{i_{b_s}}}\right)=0.
    \end{equation*}
    Let \(Y = X_{g_{i_{b_1}}} \cdots X_{g_{i_{b_s}}}\).  
    We study its diagonal entries \(Y_{pp}\).  
    From matrix multiplication,
    \begin{equation*}
        Y_{pq}
        = \sum_{k_1,\dots,k_s}
        \big(X_{g_{i_{b_1}}}\big)_{p,k_1}
        \big(X_{g_{i_{b_2}}}\big)_{k_1,k_2}\cdots
        \big(X_{g_{i_{b_s}}}\big)_{k_{s-1},q}.
    \end{equation*}
    Since \((X_{g_{i_{b_j}}})_{u,v} \neq 0\) only if  
    \(g_v = g_u g_{i_{b_j}}\),  
    for a summand to be nonzero there must exist indices 
    \(k_0,\dots,k_s\) such that
    \[
    g_{k_0}=g_p,\qquad 
    g_{k_s}=g_q,\qquad
    g_{k_j}=g_{k_{j-1}} g_{i_{b_j}}
    \quad (j=1,\dots,s).
    \]
    Iterating,
    \[
        g_p\, g_{i_{b_1}} g_{i_{b_2}} \cdots g_{i_{b_s}} = g_q.
    \]
    In particular, for a diagonal entry \(Y_{p,p}\) we need \(q=p\).  
    Thus the condition becomes
    \[
        g_p\, g_{i_{b_1}} g_{i_{b_2}} \cdots g_{i_{b_s}} = g_p
        \quad\Longrightarrow\quad
        g_{i_{b_1}} g_{i_{b_2}} \cdots g_{i_{b_s}} = e.
    \]
    Since by hypothesis  
    \(g_{i_{b_1}} g_{i_{b_2}} \cdots g_{i_{b_s}} \neq e\),
    we conclude that
    \(
        Y_{p,p} = 0,
    \) for all $p$. Hence 
    \[
        \varphi\Big(X_{g_{i_{b_1}}}\cdots X_{g_{i_{b_s}}}\Big)
        = \varphi(Y) = 0.
    \]
    Therefore, in the product defining \(\varphi_\pi\),
    one of the factors is zero, so
    \[
        \varphi_\pi[X_{g_{i_1}},\dots,X_{g_{i_r}}] = 0
        \qquad \text{for all } \pi\in \mathcal{NC}(r).
    \]
    Substituting into the moment–cumulant formula, we obtain
    \(
      \kappa_r(X_{g_{i_1}},\dots,X_{g_{i_r}}) = 0.
    \)
\end{proof}
\begin{cor}\label{cor_X}
    Let $X \in M_{|G|}(\mathcal{A})$ and define
    \begin{equation*}
        \widetilde{X}_{g_k}
        := X_{g_k^{-1}}
        = \sum_{h \in G} E_{h,h}\, X\, E_{g_k^{-1}h,\,g_k^{-1}h},\qquad k=0,2,\ldots,|G|-1.
    \end{equation*}
    Then, for ${i_1}, {i_2}, \dots, {i_r} \in \{0,1,\dots, |G|-1\}$,
    \begin{equation*}
        \kappa_r(\widetilde{X}_{g_{i_r}}, \widetilde{X}_{g_{i_{r-1}}}, \dots, \widetilde{X}_{g_{i_1}}) = 0
    \end{equation*}
    whenever $g_{i_1} g_{i_2} \cdots g_{i_r} \neq e$.
\end{cor}

\begin{proof}
    Suppose that $g_{i_1} g_{i_2} \cdots g_{i_r} \neq e$.  
    Then
    \[
        g_{i_r}^{-1} g_{i_{r-1}}^{-1} \cdots g_{i_1}^{-1} \neq e.
    \]
    Hence, using the previous lemma,
    \begin{equation*}
        \kappa_r(\widetilde{X}_{g_{i_r}}, \widetilde{X}_{g_{i_{r-1}}}, \dots, \widetilde{X}_{g_{i_1}})
        =
        \kappa_r(X_{g_{i_r}^{-1}}, X_{g_{i_{r-1}}^{-1}}, \dots, X_{g_{i_1}^{-1}})
        = 0.
    \end{equation*}
\end{proof}

\begin{lemma}\label{lemma-U-circ}
    Let $U_{e}^\ast,U_{g_1}^\ast, \dots,U_{g_{|G|-1}}^\ast$ be the left regular representation of a finite group $G=\{e,g_1,\ldots,g_{|G|-1}\}$. Then
    \begin{enumerate}
        \item $\kappa_r\big(U_{g_{i_1}}^\ast, U^\ast_{g_{i_2}}, \dots, U^\ast_{g_{i_r}}\big) = 0$ whenever 
        $g_{i_1} g_{i_2} \cdots g_{i_r} \neq e$.
        \item 
        $\kappa_{2r}\big(U^\ast_{g_{i_1}} ,U^\ast_{g_{i_1}^{-1}}, U^\ast_{g_{i_2}} ,U^\ast_{g_{i_2}^{-1}}, \dots, U^\ast_{g_{i_r}} ,U^\ast_{g_{i_r}^{-1}}\big) = 0$  
        and  
        $\kappa_{2r}\big(U^\ast_{g_{i_r}^{-1}}, U^\ast_{g_{i_1}}, U^\ast_{g_{i_1}^{-1}}, U^\ast_{g_{i_2}},U^\ast_{g_{i_2}^{-1}}, \dots, U^\ast_{g_{i_r}}\big) = 0$  
        unless $g_{i_1} = g_{i_2} = \cdots = g_{i_r}$.
    \end{enumerate}
\end{lemma}
\begin{proof}
    \begin{enumerate}[leftmargin=*]
        \item The proof follows directly from Lemma~\ref{l2}, taking $X$ to be the matrix whose entries are all equal to~1.
        \item In order to simplify the notation, lets define $S_k:=U_{g_{i_k}}^\ast$. We proceed by induction on $r$.
    \medskip
    \textbf{Base case:} We show that
    \[
        \kappa_4(S_k, S_k^{-1}, S_l, S_l^{-1}) = 0
        \qquad\text{and}\qquad
        \kappa_4(S_l^{-1}, S_k, S_k^{-1}, S_l) = 0,
    \]
    unless $g_k = g_l$.  
    Assume $g_k \neq g_l$, $g_k \neq e$.  
    Using the product formula for arguments,  
    \[
        0 = \kappa_3(S_k S_k^{-1}, S_l, S_l^{-1})
        = \sum_{\substack{\pi \in \mathcal{NC}(4) \\ \pi \vee \sigma = 1_4}} 
           \kappa_\pi[S_k, S_k^{-1}, S_l, S_l^{-1}],
    \]
    where $\sigma = \{\{1,2\}, \{3\}, \{4\}\}$.  
    Expanding the sum,
    \begin{align*}
        0 ={}& 
        \kappa_4(S_k, S_k^{-1}, S_l, S_l^{-1})
        + \kappa_2\big(S_k \kappa_2(S_k^{-1}, S_l), S_l^{-1}\big)  \\
        &+ \kappa_3\big(S_k, \kappa_1(S_k^{-1}) S_l, S_l^{-1}\big)
        + \kappa_3\big(\kappa_1(S_k) S_k^{-1}, S_l, S_l^{-1}\big).
    \end{align*}
    By part (1), all terms except the first vanish, because in each one the group relation
    $g_k g_l^{-1} \neq e$, $g_k g_l g_l^{-1} \neq e$, and $g_k^{-1} g_l g_l^{-1} \neq e$
    appears.  
    Hence
    \[
        \kappa_4(S_k, S_k^{-1}, S_l, S_l^{-1}) = 0.
    \]
    Similarly,
    \[
        0 = \kappa_3(S_l^{-1}, S_k S_k^{-1}, S_l)
        = \sum_{\substack{\pi \in \mathcal{NC}(4) \\ \pi \vee \sigma = 1_4}}
        \kappa_\pi[S_l^{-1}, S_k, S_k^{-1}, S_l],
    \]
    where now $\sigma = \{\{1\}, \{2,3\}, \{4\}\}$.  
    Expanding,
    \begin{align*}
        0 ={}&
        \kappa_4(S_l^{-1}, S_k, S_k^{-1}, S_l)
        + \kappa_2\big(S_l^{-1}, S_k \kappa_2(S_k^{-1}, S_l)\big) \\
        &+ \kappa_3\big(S_l^{-1}, S_k, \kappa_1(S_k^{-1}) S_l\big)
        + \kappa_3\big(S_l^{-1} \kappa_1(S_k), S_k^{-1}, S_l\big).
    \end{align*}
    Again by part (1), all terms except the first are zero, and thus
    \[
        \kappa_4(S_l^{-1}, S_k, S_k^{-1}, S_l) = 0.
    \]
    \medskip
    Now, assume (2) holds for all $1 \le t \le r$.  
    We show it also holds for $r+1$.  
    That is,
    \[
        \kappa_{2r+2}(S_k, S_k^{-1}, S_{i_1}, S_{i_1}^{-1}, \dots, S_{i_r}, S_{i_r}^{-1}) = 0,
    \]
    and
    \[
        \kappa_{2r+2}(S_{i_r}^{-1}, S_k, S_k^{-1}, S_{i_1}, S_{i_1}^{-1}, \dots, S_{i_r}) = 0.
    \]
    Using the product formula,
    \[
        0 = \kappa_{2r+1}(S_k S_k^{-1}, S_{i_1}, S_{i_1}^{-1}, \dots, S_{i_r}, S_{i_r}^{-1})
        = \sum_{\substack{\pi \in \mathcal{NC}(2r+2) \\ \pi \vee \sigma = 1_{2r+2}}}
          \kappa_\pi[S_k, S_k^{-1}, S_{i_1}, S_{i_1}^{-1}, \dots, S_{i_r}, S_{i_r}^{-1}],
    \]
    where 
    \(
        \sigma = \big\{\{1,2\}, \{3\}, \dots, \{2r+2\}\big\}.
    \)
    The partitions $\pi_i$ with $\pi_i \vee \sigma = 1_{2r+2}$ are exactly:
    \begin{itemize}
        \item $\pi_i = 1_{2r+2}$,
        \item partitions with two blocks, one containing $1$ and another containing $2$, i.e.
        \[
            \pi_i = \big\{\{1, i+1, \dots, 2r+2\}, \{2,3,\dots,i\}\big\},
            \qquad i=2,\dots,2r+2.
            \]
        \end{itemize}
    Consider each case:
    \begin{itemize}
        \item If $i$ is even, $i=2j$, then
        \begin{align*}
            \kappa_{\pi_{2j}}\big[&S_k,S_k^{-1},...,S_{i_r},S_{i_r}^{-1}\big]\\
            =&
            \kappa_{2r-2j+1}\big(
                S_k\, \kappa_{2j-1}(S_k^{-1}, S_{i_1}, S_{i_1}^{-1},\dots,S_{i_j},S_{i_j}^{-1}),
                S_{i_{j+1}}, S_{i_{j+1}}^{-1},\dots,S_{i_r},S_{i_r}^{-1}
            \big).
        \end{align*}
        Since $g_k \neq e$, we have  
        $g_k g_{i_1} g_{i_1}^{-1} \cdots g_{i_j} g_{i_j}^{-1} \neq e$,  
        and by part (1),
        \[
            \kappa_{2j-1}(S_k^{-1}, S_{i_1}, S_{i_1}^{-1},\dots,S_{i_j}, S_{i_j}^{-1}) = 0.
        \]
        \item If $i$ is odd, $i=2j+1$, then
        \begin{align*}
            \kappa_{\pi_{2j+1}}\big[&S_k,S_k^{-1},...,S_{i_r},S_{i_r}^{-1}\big]
            \\ 
            =&
            \kappa_{2r-2j}\big(
                S_k\, \kappa_{2j}(S_k^{-1}, S_{i_1}, S_{i_1}^{-1},\dots,S_{i_j},S_{i_j}^{-1}, S_{i_{j+1}}),
                S_{i_{j+1}}^{-1}, \dots, S_{i_r}, S_{i_r}^{-1}
            \big).
        \end{align*}
        Again, since $g_k \neq e$,  
        $g_k g_{i_1} g_{i_1}^{-1} \cdots g_{i_j} g_{i_j}^{-1} \neq e$,  
        and by part (1),
        \[
            \kappa_{2j}(S_k^{-1},S_{i_1},S_{i_1}^{-1},\dots,S_{i_j},S_{i_j}^{-1},S_{i_{j+1}})=0.
        \]
    \end{itemize}
    If $g_k \neq g_{i_1} \neq \cdots \neq g_{i_r}$, then for each partition $\pi_i$,  
    at least one block $V$ satisfies
    \[
        g_{i_{b_1}} g_{i_{b_2}} \cdots g_{i_{b_s}} \neq e,
    \]
    so by part (1),
    \[
        \kappa_{\pi_i}\big[S_k,S_k^{-1},S_{i_1},S_{i_1}^{-1},...,S_{i_r},S_{i_r}^{-1}\big]= 0.
    \]
    Hence
    \[
        \kappa_{2r+2}(S_k, S_k^{-1}, S_{i_1}, S_{i_1}^{-1}, \dots, S_{i_r}, S_{i_r}^{-1}) = 0.
    \]
    The case
    \(
        \kappa_{2r+2}(S_{i_r}^{-1}, S_k, S_k^{-1}, S_{i_1}, S_{i_1}^{-1}, \dots, S_{i_r}) = 0
    \)
    is completely analogous.
    \end{enumerate}
\end{proof}

\subsection{Proof of Main Theorem}
Finally, we prove the main result. The following lemma is the key ingredient in its proof. It establishes the freeness between the matrices associated with the left regular representation of $G$ and the components of the decomposition of $X$, under the assumption that $X$ is free from $M_{|G|}(\mathbb{C})$.


\begin{lemma}\label{lema ind X y U}
    Suppose that $X \in M_{|G|}(\mathcal{A})$ is free from $M_{|G|}(\mathbb{C})$. 
    Then the family $\left\{U_{e}^*,U_{g_1}^*,...,U_{g_{|G|-1}}^*\right\}$ is free from the family 
    $\left\{X_e, X_{g_{1}}, \dots, X_{g_{{|G|-1}}}\right\}$.
\end{lemma}
\begin{proof} Using Lemma \ref{lemma-ind-libre-conjuntos}, we prove the freeness between the families
\[
\mathcal{A}_1=\{U_{e}^*,U_{g_1}^*,\dots,U_{g_{|G|-1}}^*\},
\qquad
\mathcal{A}_2=\{X_0,X_1,\dots,X_{|G|-1}\}.
\]
That is, we must show that every mixed free cumulant involving elements from both families vanishes. Let $n\ge 1$ and consider a mixed cumulant
\[
k_n(a_1,\dots,a_n), \qquad a_j\in \mathcal{A}_1\cup\mathcal{A}_2,
\]
with at least one $a_j\in\mathcal{A}_1$ and at least one $a_j\in\mathcal{A}_2$.
Write the positions where elements of $\mathcal{A}_2$ occur as
\[
J:=\{j_1<j_2<\cdots<j_r\}\subset\{1,\dots,n\},\qquad a_{j_t}=X_{g_{j_t}},\ \ r\ge 1.
\]
Using multilinearity and the decomposition $X_g=\sum_{h\in G}E_h X E_{gh}$, we obtain
\[
k_n(a_1,\dots,a_n)
=
\sum_{h_1,\dots,h_r\in G}
k_n\!\Big(a_1,\dots,a_{j_1-1},\,E_{h_1}XE_{g_{j_1}h_1},\,a_{j_1+1},\dots,
E_{h_r}XE_{g_{j_r}h_r},\dots,a_n\Big).
\]
Now apply the product as arguments formula to the previous equation.
This yields a sum over $\pi\in \mathcal{NC}(n+2r)$ satisfying $\pi\vee\sigma=1_{n+2r}$,
where $\sigma$ is the partition whose only non-singleton blocks are the $r$ triples
\[
\sigma=\big(\text{singletons}\big)\ \cup\ \big\{\{j_t',j_t'+1,j_t'+2\}: t=1,\dots,r\big\},
\]
with $j_t'$ denoting the position of the first factor $E_{h_t}$ after expansion. Hence
\[
k_n(a_1,\dots,a_n)
=
\sum_{h_1,\dots,h_r\in G}
\sum_{\substack{\pi\in \mathcal{NC}(n+2r)\\ \pi\vee\sigma=1_{n+2r}}}
k_\pi\big[\ \cdots,\ E_{h_1},X,E_{g_{j_1}h_1},\ \cdots,\ E_{h_r},X,E_{g_{j_r}h_r},\ \cdots\ \big].
\]
Since $X$ is free from $M_{|G|}(\mathbb{C})$, any mixed free cumulant involving both $X$
and an element of $M_{|G|}(\mathbb{C})$ vanishes. In particular, for a fixed
$\pi\in \mathcal{NC}(n+2r)$, the quantity
\[
k_\pi\big[\ \cdots,\ E_{h_1},X,E_{g_{j_1}h_1},\ \cdots,\ E_{h_r},X,E_{g_{j_r}h_r},\ \cdots\big],
\]
is equal to $0$ whenever $\pi$ has a block containing an occurrence of $X$ and also containing at least one
deterministic element (i.e.\ any $E_{h_t}$, $E_{g_{j_t}h_t}$, or any $U_g^*$ appearing among the other entries).
Consequently, the only partitions $\pi$ that can contribute are those for which every block intersecting the set of
$X$-positions is entirely contained in the set of $X$-positions.
\\
For such a $\pi$, write $\pi=\pi_X\sqcup \pi_M$, where $\pi_X$ is the restriction of $\pi$ to the positions
\[
\{j_1'+1,\dots,j_r'+1\}
\qquad(\text{the positions occupied by the }X\text{'s}),
\]
and $\pi_M$ is the restriction to the remaining positions
\[
M:=\{1,\dots,n+2r\}\setminus\{j_1'+1,\dots,j_r'+1\}.
\]
Then, by multiplicativity of cumulants over blocks,
\begin{align*}
k_n(a_1,&\dots,a_n)
=
\sum_{h_1,\dots,h_r\in G}
\sum_{\substack{\pi_X\in \mathcal{NC}(r)\\ \pi_M\in \mathcal{NC}(M)\\ (\pi_X\sqcup \pi_M)\vee\sigma=1_{n+2r}}}
k_{\pi_X}[X,\dots,X]\;
k_{\pi_M}\big[\ \cdots,\ E_{h_1},E_{g_{j_1}h_1},\ \cdots,\ E_{h_r},E_{g_{j_r}h_r},\ \cdots\big]
\\
&=
\sum_{\pi_X\in \mathcal{NC}(r)}k_{\pi_X}[X,\dots,X]\;
\sum_{h_1,\dots,h_r\in G}
\sum_{\substack{\pi_M\in \mathcal{NC}(M)\\ (\pi_X\sqcup \pi_M)\vee\sigma=1_{n+2r}}}
k_{\pi_M}\big[\ \cdots,\ E_{h_1},E_{g_{j_1}h_1},\ \cdots,\ E_{h_r},E_{g_{j_r}h_r},\ \cdots\big].
\end{align*}
Let $\widehat{\pi}_M$ be the largest non-crossing partition of $M$ such that $\pi_X\sqcup \widehat{\pi}_M\in \mathcal{NC}(n+2r)$. 
Then any $\pi_M$ contributing to the sum satisfies $\pi_M\le \widehat{\pi}_M$. Write $\widehat{\pi}_M=\{B^{(1)},\dots,B^{(\ell)}\}$.  
For each block $B^{(i)}$, there exist indices
\(
1\leq i(1)<\cdots<i(s_i)\leq r,
\)
such that $B^{(i)}$ has the explicit form
\begin{equation*}
B^{(i)}
=
\Big\{
j_{i(1)}'+2,\dots,j_{i(2)}',
\;
j_{i(2)}'+2,\dots,j_{i(3)}',
\;
\dots,
\;
j_{i(s_i)}'+2,\dots,j_{i(1)}'
\Big\}.
\end{equation*}
In particular:
\begin{itemize}
\item The entries at positions $j_{i(k)}'$ are of the form $E_{h_{i(k)}}$.
\item The entries at positions $j_{i(k)}'+2$ are of the form $E_{g_{j_{i(k)}}h_{i(k)}}$.
\item The remaining entries inside the block correspond to elements of the form $U_g^*$.
\end{itemize}
Moreover, the elements at positions $j_{i(k)}'+1$ and $j_{i(k+1)}'+1$ correspond to $X$'s that lie in the same block of $\pi_X$ (cyclically, with $i(s_i+1)=i(1)$). By the multilinearity of each block
\begin{align*}
\sum_{\substack{\pi_M\in \mathcal{NC}(M)\\(\pi_X\sqcup \pi_M)\vee\sigma=1_{n+2r}}}
&k_{\pi_M}[\cdots]
=
\prod_{i=1}^{\ell}
\left(
\sum_{\substack{\rho\in \mathcal{NC}(B^{(i)})\\ (\pi_X\sqcup\rho)\vee\sigma|_{B^{(i)}}=1}}
k_{\rho}\big[\text{Entries restricted to }B^{(u)}\big]
\right)\\ 
=&\prod_{i=1}^{\ell}
\left(
\sum_{\substack{\rho\in \mathcal{NC}(B^{(i)})\\ (\pi_X\sqcup\rho)\vee\sigma|_{B^{(i)}}=1}}
k_{\rho}\Big[E_{g_{j_{i(1)}}h_{i(1)}},...,E_{h_{i(2)}},E_{g_{j_{i(3)}}h_{i(3)}},...,E_{h_{i(s-2)}},E_{g_{j_{i(s-1)}}h_{i(s-1)}}...,E_{h_{i(s)}}\Big]\right).
\end{align*}
Define $\sigma_{B^{(i)}}\in \mathcal{NC}(B^{(i)})$ as the partition whose only non-singleton blocks are
\[
\{j_{i(1)}',j_{i(2)}'+2\},
\;
\{j_{i(2)}',j_{i(3)}'+2\},
\;
\dots,
\;
\{j_{i(s_i)}',j_{i(1)}'+2\},
\]
and all remaining points are singletons. Then the condition 
\(
(\pi_X\sqcup\rho)\vee\sigma|_{B^{(i)}}=1
\) is equivalent to the condition
\(
\rho\vee\sigma_{B^{(i)}}=1_{B^{(i)}}
\). Then 
\begin{align*}
    \sum_{\substack{\pi_M\in \mathcal{NC}(M)\\(\pi_X\sqcup \pi_M)\vee\sigma=1_{n+2r}}}
&k_{\pi_M}[\cdots]
\\ 
=&\prod_{i=1}^{\ell}
\left(
\sum_{\substack{\rho\in \mathcal{NC}(B^{(i)})\\ \rho\vee\sigma_{B^{(i)}}=1_{B^{(i)}}}}
k_{\rho}\Big[E_{g_{j_{i(1)}}h_{i(1)}},...,E_{h_{i(2)}},E_{g_{j_{i(3)}}h_{i(3)}},...,E_{h_{i(s-2)}},E_{g_{j_{i(s-1)}}h_{i(s-1)}}...,E_{h_{i(s)}}\Big]\right).
\end{align*}
Applying the product as arguments formula (collapsing the pairs prescribed by $\sigma_B$), the block contribution
can be rewritten in terms of a cumulant whose entries include products of the form
\[
E_{h_{i(k)}}\,E_{g_{j_{i(k+1)}}h_{i(k+1)}},\qquad k=1,\dots,s
\quad(\text{with }i(s+1)=i(1)).
\]
Using
\[
E_aE_b=\delta_{a,b}E_a,\qquad a,b\in G,
\]
we obtain
\[
E_{h_{i(k)}}E_{g_{j_{i(k+1)}}h_{i(k+1)}}
=
\delta_{\,h_{i(k)},\,g_{j_{i(k+1)}}h_{i(k+1)}}\;E_{h_{i(k)}}.
\]
If any of these Kronecker deltas is zero, then the corresponding cumulant term vanishes and does not contribute.
Hence it suffices to consider the case when all deltas are equal to $1$. In that case, the constraints force
\[
h_{i(1)}=g_{j_{i(2)}}h_{i(2)},\quad
h_{i(2)}=g_{j_{i(3)}}h_{i(3)},\ \dots,\ 
h_{i(s)}=g_{j_{i(1)}}h_{i(1)}.
\]
Substituting recursively, we obtain
\[
h_{i(1)}
=
g_{j_{i(2)}}g_{j_{i(3)}}\cdots g_{j_{i(s)}}g_{j_{i(1)}}\,h_{i(1)}.
\]
In particular, this implies the cyclic constraint
\[
g_{j_{i(2)}}g_{j_{i(3)}}\cdots g_{j_{i(s)}}g_{j_{i(1)}}=e,
\]
and moreover all $h_{i(k)}$ are determined by a single free parameter $h:=h_{i(1)}$, i.e.\
$h_{i(k)}=\alpha_k h$ for suitable $\alpha_k\in G$.
\\
Now we use the key observation: after imposing the delta constraints, exactly one free summation parameter $h\in G$ remains in the block. By multilinearity of free cumulants, we may sum over one $E$-entry
\[
\sum_{h\in G}E_{\alpha h}=\sum_{t\in G}E_t=I.
\]
Since the block is non-trivial (it connects at least two positions), the resulting cumulant has order at least $2$, and therefore
\[
k_m(\dots,I,\dots)=0, \qquad \text{for all } m\ge2.
\]
Thus each block factor vanishes after summing over the $h$'s, and consequently
\(
k_n(a_1,\dots,a_n)=0,
\)
for every mixed cumulant. This proves that $\Big\{U_{e}^*,U_{g_1}^*,...,U_{g_{|G|-1}}^*\Big\}$ is free from $\Big\{X_0, X_1, \dots, X_{|G|-1}\Big\}$.
\end{proof}
\begin{thm}\label{main-thm}
     Let $G=\{e,g_1,...,g_{|G|-1}\}$ be a finite group. Assume that $X\in M_{|G|}(\mathcal{A})$ is free from $M_{|G|}(\mathbb{C})$ over  $\mathbb{C}$. Let 
    $$X^t=Y_e+Y_{g_1}+\cdots+Y_{g_{{|G|-1}}},$$ be the left $G$-circulant decomposition from $X^t$. The 
    \begin{equation*}
        Y_{e},\{Y_{g},Y_{g^{-1}}\},\text{ for all }g\in G\backslash \{e\},
    \end{equation*}
    is a free family over $\mathbb{C}$. 
    Moreover, $\{Y_g,Y_{g^{-1}}\}$ is $R$-diagonal pair for all $g\in G\backslash\{e\}$.  
\end{thm}
\begin{proof}
    By Lemma \ref{lemma Y-UXU}, we obtain
\begin{equation*}
    Y_{g_k}
    =
    U_{g_k}^*
    \left[
        \sum_{h \in G}
        E_{h,h}\, X\, E_{g_k^{-1}h,\,g_k^{-1}h}
    \right]
    U_{g_k}^*
    =
    U_{g_k}^*\widetilde{X}_{g_k}U_{g_k}^*,
    \qquad
    \forall\, g_k \in \{e,g_1,\ldots,g_{|G|-1}\}.
\end{equation*}
Let $r>0$ and let $h_{1},\ldots,h_{r}\in\{e,g_1,\ldots,g_{|G|-1}\}$. Assume that the indices
$h_1,\ldots,h_r$ do not all belong to the same class of the partition
\[
\{e\},\qquad \{g,g^{-1}\},\quad g\in G\setminus\{e\}.
\]
We shall prove that
\begin{equation*}
    \kappa_r(Y_{h_{1}},Y_{h_{2}},\ldots,Y_{h_{r}})=0.
\end{equation*}
We distinguish the following two cases.
\begin{enumerate}
\item We have $h_j\neq e$ for every $j=1,\ldots,r$. In this case,
\(
Y_{h_j}=U_{h_j}^*\widetilde{X}_{h_j}U_{h_j}^*,
\)
and each term $Y_{h_j}$ contributes three factors.
\item  There exists some $1\leq j\leq r$ such that $h_j=e$. Then
\(
Y_{h_j}
=
U_{e}^*\widetilde{X}_{e}U_{e}^*
=
\widetilde{X}_{e}
=
X_e,
\)
that is, the term $Y_{h_j}$ contributes only a single factor.
\end{enumerate}
We now treat each of the two cases separately.

 \textbf{Case (1)}: Assume that $h_{1},\ldots,h_{r}\neq e$. That is, each $Y_{h_k}$ can be written as a product of three factors of the form
\begin{equation*}
    Y_{h_k}
    =
    U_{h_k}^*
    \left[
        \sum_{l \in G}
        E_{l,l}\, X\, E_{h_k^{-1}l,\,h_k^{-1}l}
    \right]
    U_{h_k}^*
    =
    U_{h_k}^*\widetilde{X}_{h_k}U_{h_k}^*,
    \qquad
    k=1,2,\ldots,r.
\end{equation*}
By the product formula for cumulants, we have
\begin{equation*}
    \kappa_r(Y_{h_{1}},Y_{h_{2}},\ldots,Y_{h_{r}})
    =
    \sum_{\substack{\pi\in \mathcal{NC}(3r)\\\pi\vee\sigma=1_{3r}}}
    \kappa_\pi
    \big[
        U_{h_{1}}^*,\widetilde{X}_{h_1},U_{h_{1}}^*,
        U_{h_{2}}^*,\widetilde{X}_{h_2},U_{h_{2}}^*,
        \ldots,
        U_{h_{r}}^*,\widetilde{X}_{h_r},U_{h_{r}}^*
    \big],
\end{equation*}
where $\sigma=\{V_1,\ldots,V_r\}$ with
$V_j=\{3j-2,3j-1,3j\}$ for each $j=1,2,\ldots,r$. Observe that each triple
\(
(U_{h_i}^*,\widetilde{X}_{h_i},U_{h_i}^*)
\) 
is associated with the triple of points
\(
(3i-2,\,3i-1,\,3i),
\)
for every $i=1,2,\ldots,r$. By Lemma \ref{lema ind X y U}, any partition
$\pi\in \mathcal{NC}(3r)$ satisfying $\pi\vee\sigma=1_{3r}$ and containing a block with both $X$-variables and $U$-variables simultaneously yields
\(
\kappa_\pi\big[
U_{h_{1}}^*,\widetilde{X}_{h_1},U_{h_{1}}^*,\ldots,U_{h_{r}}^*,\widetilde{X}_{h_r},U_{h_{r}}^*\big]=0.
\)
Therefore, in order for
\(\kappa_\pi\big[U_{h_{1}}^*,\widetilde{X}_{h_1},U_{h_{1}}^*,\ldots,U_{h_{r}}^*,\widetilde{X}_{h_r},U_{h_{r}}^*\big]
\neq 0,
\)
every block of $\pi$ must contain exclusively variables of type $X$ or exclusively variables of type $U$. Now consider a partition $\pi\in \mathcal{NC}(3r)$ with $\pi\vee\sigma=1_{3r}$ such that
\begin{equation}
\label{eq_cum_no_0}
\kappa_\pi\big[
U_{h_1}^*,\widetilde X_{h_1},U_{h_1}^*,
\dots, U_{h_r}^*,\widetilde X_{h_r},U_{h_r}^* \big]\neq 0.
\end{equation}
Let $V\in\pi$ be a $U$-block. Without loss of generality, assume that the point
\(
b:=3i
\)
belongs to $V$. The block $V$ cannot be a singleton, since otherwise Lemma \ref{lemma-U-circ} would imply that
\( 
\kappa_1(U_{h_i}^\ast)=0,
\)
contradicting \eqref{eq_cum_no_0}.

Let $b^{(1)}\in V$ denote the next element of the same block lying to the right of $b$ (if no such element exists and the next element lies to the left, the argument is analogous). There are two possibilities:
\[
b^{(1)}=3j-2
\qquad\text{or}\qquad
b^{(1)}=3j,
\]
for some $j>i$. Notice that if $b^{(1)}=3j-2$, then the condition
$\pi\vee\sigma=1_{3r}$ is violated unless $j=i+1$. Assume first that $b^{(1)}=3j$, as illustrated in the following diagram:
\[
\begin{tikzpicture}[scale=0.8]
    \def\dx{1.3}
    \def\drop{1}

    \foreach \k in {0,1,2,3,4,5,7,8,9,10,11}
        \fill (\k*\dx,0) circle (1.5pt);

    \node[above=3pt] at (0*\dx,0) {$U_{h_i}^*$};
    \node[above=3pt] at (1*\dx,0) {$\widetilde{X}_{h_i}$};
    \node[above=3pt] at (2*\dx,0) {$U_{h_i}^*$};

    \node[above=3pt] at (2*\dx,-2) {$b$};
    \node[above=3pt] at (11*\dx,-2) {$b^{(1)}$};
    
    \node[above=3pt] at (3*\dx,0) {$U_{h_{i+1}}^*$};
    \node[above=3pt] at (4*\dx,0) {$\widetilde{X}_{h_{i+1}}$};
    \node[above=3pt] at (5*\dx,0) {$U_{h_{i+1}}^*$};

    \node at (6*\dx,0.15) {$\cdots$};

    \node[above=3pt] at (7*\dx,0) {$\widetilde{X}_{h_{j-1}}$};
    \node[above=3pt] at (8*\dx,0) {$U_{h_{j-1}}^*$};
    \node[above=3pt] at (9*\dx,0) {$U_{h_j}^*$};
    \node[above=3pt] at (10*\dx,0) {$\widetilde{X}_{h_j}$};
    \node[above=3pt] at (11*\dx,0) {$U_{h_j}^*$};

    \draw[thick] (2*\dx,0) -- ++(0,-\drop);
    \draw[thick] (11*\dx,0) -- ++(0,-\drop);
    \draw[thick] (2*\dx,-\drop) -- ++(-1,0);
    \draw[thick] (2*\dx,-\drop) -- (11*\dx,-\drop);
    \draw[thick] (11*\dx,-\drop) -- ++(1,0);

    \draw[red,dashed] ({2.5*\dx},1) -- ({2.5*\dx},-0.8);
    \draw[red,dashed] ({5.5*\dx},1) -- ({5.5*\dx},-0.8);
    \draw[red,dashed] ({8.5*\dx},1) -- ({8.5*\dx},-0.8);

\end{tikzpicture}
\]
The following three lemmas are importants in the rest of the proof:
\begin{lemma}\label{lema-1-dem}
Suppose that an $X$-block is an interval. Then such an interval cannot have odd cardinality.
\end{lemma}

\begin{proof}
Assume for contradiction, that the $X$-block is an interval of odd cardinality of the form
\[
\big\{
3b_1-1,,
3(b_1+1)-1,,
3(b_1+2)-1,,
\ldots,,
3(b_1+2t)-1
\big\},
\]
for some $t\geq 0$. Such a configuration is represented in the following diagram:
\[
\begin{tikzpicture}[scale=1.1]
    \def\dx{1.2}
    \def\drop{1}

    \foreach \k in {0,1,2,3,4,5,7,8,9}
        \fill (\k*\dx,0) circle (1.5pt);

    
    \node[above=3pt] at (0*\dx,0) {$\widetilde X_{h_{b_1}}$};
    \node[above=3pt] at (1*\dx,0) {$U^\ast_{h_{b_1}}$};
    \node[above=3pt] at (2*\dx,0) {$U^\ast_{h_{b_1+1}}$};

    \node[above=3pt] at (3*\dx,0) {$\widetilde X_{h_{b_1+1}}$};
    \node[above=3pt] at (4*\dx,0) {$U^\ast_{h_{b_1+1}}$};
    \node[above=3pt] at (5*\dx,0) {$U^\ast_{h_{b_1+2}}$};

    \node at (6*\dx,0.15) {$\cdots$};

    \node[above=3pt] at (7*\dx,0) {$U^\ast_{h_{b_1+2t-1}}$};
    \node[above=3pt] at (8*\dx,0) {$U^\ast_{h_{b_1+2t}}$};
    \node[above=3pt] at (9*\dx,0) {$\widetilde X_{h_{b_1+2t}}$};

    \draw[thick] (0*\dx,0) -- ++(0,-\drop);
    \draw[thick] (3*\dx,0) -- ++(0,-\drop);
    \draw[thick] (9*\dx,0) -- ++(0,-\drop);
    \draw[thick] (0*\dx,-\drop) -- (9*\dx,-\drop);

    \draw[thick] (1*\dx,0) -- ++(0,-0.6);
    \draw[thick] (2*\dx,0) -- ++(0,-0.6);
    \draw[thick] (1*\dx,-0.6) -- (2*\dx,-0.6);

    \draw[thick] (4*\dx,0) -- ++(0,-0.6);
    \draw[thick] (5*\dx,0) -- ++(0,-0.6);
    \draw[thick] (4*\dx,-0.6) -- (5*\dx,-0.6);

    \draw[thick] (7*\dx,0) -- ++(0,-0.6);
    \draw[thick] (8*\dx,0) -- ++(0,-0.6);
    \draw[thick] (7*\dx,-0.6) -- (8*\dx,-0.6);

    \draw[red,dashed] ({1.5*\dx},1) -- ({1.5*\dx},-0.8);
    \draw[red,dashed] ({4.5*\dx},1) -- ({4.5*\dx},-0.8);
    \draw[red,dashed] ({7.5*\dx},1) -- ({7.5*\dx},-0.8);
\end{tikzpicture}
\]
Since we are assuming that
\(\kappa_\pi[\cdots]\neq 0,
\)
it follows that
\[
\kappa_2\Big(U^\ast_{h_{b_1}},U^\ast_{h_{b_1+1}}\Big)\neq 0,
\qquad
\kappa_2\Big(U^\ast_{h_{b_1+1}},U^\ast_{h_{b_1+2}}\Big)\neq 0,
\qquad \ldots, \qquad
\kappa_2\Big(U^\ast_{h_{b_1+2t-1}},U^\ast_{h_{b_1+2t}}\Big)\neq 0,
\]
and
\[
\kappa_{2t+1}
\Big(
\widetilde X_{h_{b_1}},
\widetilde X_{h_{b_1+1}},
\ldots,
\widetilde X_{h_{b_1+2t}}
\Big)
\neq 0.
\]
By Lemma \ref{lemma-U-circ} and Corollary \ref{cor_X}, we obtain
\[
h_{b_1}h_{b_1+1}=e,
\qquad
h_{b_1+1}h_{b_1+2}=e,
\qquad
\ldots ,
\qquad
h_{b_1+2t-1}h_{b_1+2t}=e,
\]
and
\[
h_{b_1+2t}h_{b_1+2t-1}\cdots h_{b_1+2}h_{b_1+1}h_{b_1}=e.
\]
Combining these relations yields
\(
h_{b_1+2t}=e,
\)
which is impossible. This contradiction shows that an $X$-block that forms an interval must have even cardinality.
\end{proof}
Moreover, the previous lemma shows that if the $X$-block is of the form
\[
\big\{
3b_1-1,
3(b_1+1)-1,
3(b_1+2)-1,
3(b_1+3)-1,
\ldots,
3(b_1+2t)-1,
3(b_1+2t+1)-1
\big\},
\]
that is, an interval $X$-block of even cardinality, then the corresponding variables are necessarily
\[
\Big\{
X_{h_{b_1}},
X_{h_{b_1}^{-1}},
X_{h_{b_1}},
X_{h_{b_1}^{-1}},
\ldots,
X_{h_{b_1}},
X_{h_{b_1}^{-1}}
\Big\}.
\]
\begin{lemma}\label{lema-2-dem}
Let $V$ be an $X$-block (not necessarily an interval). If $3a-1$ and $3c-1$, with $a<c$, are consecutive elements of $V$, then they correspond to the variables $X_h$ and $X_{h^{-1}}$, respectively, for some $h\in G$.
\end{lemma}

\begin{proof}
The proof proceeds by induction on the number of nested blocks. First, suppose that the consecutive elements $3a-1$ and $3c-1$ contain exactly one nested block. In this case, the nested block corresponds to the interval generated by the elements $U_{h_a}^\ast$ and $U_{h_c}^\ast$, as illustrated in the following diagram:

\[
\begin{tikzpicture}[scale=0.8]
\def\dx{1.55}
\def\drop{0.8}
\foreach \k in {0,...,3}
\fill (\k*\dx,0) circle (1.5pt);
\node[above=3pt] at (0*\dx,0) {$\widetilde X_{h_a}$};
\node[above=3pt] at (1*\dx,0) {$U^\ast_{h_a}$};
\node[above=3pt] at (2*\dx,0) {$U_{h_c}^\ast$};
\node[above=3pt] at (3*\dx,0) {$\widetilde X_{h_c}$};

```
\draw[thick] (0*\dx,0) -- ++(0,-\drop);
\draw[thick] (3*\dx,0) -- ++(0,-\drop);
\draw[thick] ($(0*\dx,-\drop)+(-1,0)$) -- ($(3*\dx,-\drop)+(1,0)$);

\draw[thick] (1*\dx,0) -- ++(0,-0.5);
\draw[thick] (2*\dx,0) -- ++(0,-0.5);
\draw[thick] (1*\dx,-0.5) -- (2*\dx,-0.5);

\draw[red,dashed] ({1.5*\dx},1) -- ({1.5*\dx},-0.8);
```

\end{tikzpicture}
\]
Since we are assuming that
\(
\kappa_\pi[\cdots]\neq 0,
\)
it follows that
\[
\kappa_2(U_{h_a}^*,U_{h_c}^*)\neq 0 \quad \Rightarrow\quad h_ah_c=e\quad \Rightarrow\quad h_c=h_a^{-1}.
\]
Assume now that the statement holds whenever there are between $1$ and $s$ nested blocks between the elements $3a-1$ and $3c-1$. We shall prove it in the case where there are $s+1$ nested blocks.

Observe that if there is an $X$-block nested between the elements $3a-1$ and $3c-1$, then, by the induction hypothesis, the consecutive elements of this nested $X$-block correspond to the variables $\widetilde X_h$ and $\widetilde X_{h^{-1}}$, respectively. Without loss of generality, suppose that there are $t$ nested $X$-blocks of this type.

The $m$-th such block has the form illustrated in the following diagram:
\[
\begin{tikzpicture}[scale=0.75,transform shape]
    \def\dx{1.3}
    \def\drop{1}
    \def\bigdrop{1.55}

    \foreach \k in {0,1,2,4,5,6,7,9,10,11,12,14,15,16}
        \fill (\k*\dx,0) circle (1.5pt);

\fill (6.25*\dx,0) circle (1.5pt);
\fill (6.75*\dx,0) circle (1.5pt);
    \draw[thick] (6.25*\dx,0) -- ++(0,-0.6);
    \draw[thick] (6.75*\dx,0) -- ++(0,-0.6);
    \draw[thick] (6.25*\dx,-0.6) -- (6.75*\dx,-0.6);

\fill (9.25*\dx,0) circle (1.5pt);
\fill (9.75*\dx,0) circle (1.5pt);
    \draw[thick] (9.25*\dx,0) -- ++(0,-0.6);
    \draw[thick] (9.75*\dx,0) -- ++(0,-0.6);
    \draw[thick] (9.25*\dx,-0.6) -- (9.75*\dx,-0.6);

    \node[above=3pt] at (0*\dx,0) {$\widetilde X_{h_a}$};

    \node[above=3pt] at (1*\dx,0) {$U_{h_a}^\ast$};
    \node[above=3pt] at (2*\dx,0) {$U_{h_{a_1}}^\ast$};

    \draw[thick] (1*\dx,0) -- ++(0,-0.6);
    \draw[thick] (2*\dx,0) -- ++(0,-0.6);
    \draw[thick] (1*\dx,-0.6) -- (2*\dx,-0.6);

    \node at (3*\dx,0.15) {$\cdots$};

    \node[above=3pt] at (4*\dx,0) {$U_{h_{a_{m-1}}^{-1}}^\ast$};
    \node[above=3pt] at (5*\dx,0) {$U_{h_{a_m}}^\ast$};

    \draw[thick] (4*\dx,0) -- ++(0,-0.6);
    \draw[thick] (5*\dx,0) -- ++(0,-0.6);
    \draw[thick] (4*\dx,-0.6) -- (5*\dx,-0.6);

    \node[above=3pt] at (6*\dx,0) {$\widetilde X_{h_{a_m}}$};
    \node[above=3pt] at (7*\dx,0) {$\widetilde X_{h_{a_m}^{-1}}$};

    \node at (8*\dx,0.15) {$\cdots$};

    \node[above=3pt] at (9*\dx,0) {$\widetilde X_{h_{a_m}}$};
    \node[above=3pt] at (10*\dx,0) {$\widetilde X_{h_{a_m}^{-1}}$};

    \node[above=3pt] at (11*\dx,0) {$U_{h_{a_m}^{-1}}^\ast$};
    \node[above=3pt] at (12*\dx,0) {$U_{h_{a_m+1}}^\ast$};

    \draw[thick] (11*\dx,0) -- ++(0,-0.6);
    \draw[thick] (12*\dx,0) -- ++(0,-0.6);
    \draw[thick] (11*\dx,-0.6) -- (12*\dx,-0.6);

    \node at (13*\dx,0.15) {$\cdots$};

    \node[above=3pt] at (14*\dx,0) {$U_{h_{a_t}^{-1}}^\ast$};
    \node[above=3pt] at (15*\dx,0) {$U_{h_{c}}^\ast$};

    \draw[thick] (14*\dx,0) -- ++(0,-0.6);
    \draw[thick] (15*\dx,0) -- ++(0,-0.6);
    \draw[thick] (14*\dx,-0.6) -- (15*\dx,-0.6);

    \node[above=3pt] at (16*\dx,0) {$\widetilde X_{h_c}$};

    \foreach \k in {6,7,9,10}
        \draw[thick] (\k*\dx,0) -- ++(0,-\drop);
    \draw[thick] (6*\dx,-\drop) -- (10*\dx,-\drop);

    \draw[thick] (0*\dx,0) -- ++(0,-\bigdrop);
    \draw[thick] (16*\dx,0) -- ++(0,-\bigdrop);
    \draw[thick] (0*\dx,-\bigdrop) -- (16*\dx,-\bigdrop);

    \draw[red,dashed] ({1.5*\dx},1.3) -- ({1.5*\dx},-1.3);
    \draw[red,dashed] ({4.5*\dx},1.3) -- ({4.5*\dx},-1.3);
    \draw[red,dashed] ({6.5*\dx},1.3) -- ({6.5*\dx},-1.3);

\draw[red,dashed] ({11.5*\dx},1.3) -- ({11.5*\dx},-1.3);
\draw[red,dashed] ({14.5*\dx},1.3) -- ({14.5*\dx},-1.3);
\draw[red,dashed] ({9.5*\dx},1.3) -- ({9.5*\dx},-1.3);

\end{tikzpicture}
\]
Since we are assuming that
\(
\kappa_\pi[\cdots]\neq 0,
\)
every $U$-block contained between the elements $3a-1$ and $3c-1$ must have a non-vanishing associated cumulant. 

We now introduce a convenient terminology. Given a sequence
\(
U_{g_1},U_{g_2},\dots,U_{g_n},
\)
and a partition $\pi$ of $\{1,\dots,n\}$, we say that the indices
\[
g_1,g_2,\dots,g_n
\]
are multiplied \emph{according to the nesting structure of $\pi$} if the product associated with each block is evaluated recursively, beginning with the innermost blocks and proceeding outward. More precisely, if $V=\{i_1<\cdots<i_r\}$ is a block of $\pi$, we associate to $V$ the product
\[
g_{i_1}g_{i_2}\cdots g_{i_r}.
\]
Whenever $V$ contains inner blocks of $\pi$, these are evaluated first and their values are substituted into the product corresponding to $V$. In this way, the final evaluation is completely determined by the nesting structure of the partition.

\begin{Example}
Consider the sequence
\[
U_{h_a}\,,
U_{h_{a_1}}\,,
U_{h_{a_1}^{-1}}\,,
U_{h_{a_2}}\,,
U_{h_{a_2}^{-1}}\,,
U_{h_{a_3}}\,,
U_{h_{a_3}^{-1}}\,,
U_{h_c},
\]
and the partition
\(
\pi=\big\{\{1,2,7,8\},\{3,4\},\{5,6\}\big\}.
\)
The blocks $\{3,4\}$ and $\{5,6\}$ are inner blocks, while $\{1,2,7,8\}$ is the outer block. To each block we associate the product of the indices appearing in the corresponding positions:
\[
\{3,4\}
\longmapsto
h_{a_1}^{-1}h_{a_2},
\qquad
\{5,6\}
\longmapsto
h_{a_2}^{-1}h_{a_3},
\qquad
\text{and}\qquad
\{1,2,7,8\}
\longmapsto
h_a,h_{a_1},h_{a_3}^{-1},h_c.
\]
Suppose that
\[
h_{a_1}^{-1}h_{a_2}=e,
\qquad
h_{a_2}^{-1}h_{a_3}=e,
\qquad
h_a,h_{a_1},h_{a_3}^{-1},h_c=e.
\]
We first evaluate the inner blocks
\[
h_{a_1}^{-1}h_{a_2}=e,
\qquad
h_{a_2}^{-1}h_{a_3}=e.
\]
Substituting these identities into the total product yields
\[
h_a\,h_{a_1}\,
\bigl(h_{a_1}^{-1}h_{a_2}\bigr)\,
\bigl(h_{a_2}^{-1}h_{a_3}\bigr)\,
h_{a_3}^{-1}\,h_c
=
h_a\,h_{a_1}\,e\,e\,h_{a_3}^{-1}\,h_c
=
h_a\,h_{a_1}\,h_{a_3}^{-1}\,h_c.
\]
Finally, using the condition associated with the outer block, we obtain
\(
h_a\,h_{a_1}\,h_{a_3}^{-1}\,h_c=e.
\)
Hence, the total product of the indices, evaluated according to the nesting structure of the partition $\pi$, is equal to the identity element.
\end{Example}
To complete the proof, consider the sequence of $U$-variables lying between the elements $3a-1$ and $3c-1$, namely
\[
U_{h_a}^\ast,
U_{h_{a_1}}^\ast,
\cdots,
U_{h_{a_{m-1}}^{-1}}^\ast,
U_{h_{a_m}}^\ast,
U_{h_{a_m}^{-1}}^\ast,
U_{h_{a_m+1}}^\ast,
\cdots,
U_{h_{a_t}^{-1}}^\ast,
U_{h_c}^\ast.
\]
Multiplying the corresponding indices according to the nesting structure yields
\[
h_a
h_{a_1}
\cdots
h_{a_{m-1}}^{-1}
h_{a_m}
h_{a_m}^{-1}
h_{a_m+1}
\cdots
h_{a_t}^{-1}
h_c
=e.
\]
Since every nested block contributes the identity, all intermediate factors cancel, and we obtain
\(
h_ah_c=e.
\)
Therefore,
\(
h_c=h_a^{-1}.
\)
This completes the induction and hence the proof of the lemma.

\end{proof}

\begin{lemma}
Let
\begin{equation*}
    B^{(i,j)}
    =
    \big\{
    3(i+1)-2,3(i+1)-1,3(i+1),
    \ldots,
    3(j-1)-2,3(j-1)-1,3(j-1),
    3j-2,3j-1
    \big\}.
\end{equation*}
Then
\(
|B^{(i,j)}| = b^{(1)}-b-1.
\)
If
\(
\widehat{\pi}\in \mathcal{NC}\big(B^{(i,j)}\big)
\cong \mathcal{NC}\big(b^{(1)}-b-1\big),
\)
then
\begin{equation}
\kappa_{\widehat{\pi}}
\big[
U_{h_{i+1}}^*,
\widetilde{X}_{h_{i+1}},
U_{h_{i+1}}^*,
\dots,
\widetilde{X}_{h_{j-1}},
U_{h_{j-1}}^*,
U_{h_j}^*,
\widetilde{X}_{h_j}
\big]
=0.
\end{equation}
\end{lemma}
\begin{proof}
    The proof proceeds by induction on
\(
l:=\frac{b^{(1)}-b-3}{3}.
\)
For the base case $l=1$, we have $b=3i$ and $b^{(1)}=3(i+2)$. Thus, we are in the situation represented by the following diagram
\[
\begin{tikzpicture}[scale=0.8]
\def\dx{1.5}
\def\drop{1}

```
\foreach \k in {0,...,8}
    \fill (\k*\dx,0) circle (1.5pt);

\node[above=3pt] at (0*\dx,0) {$U_{h_{i}}^*$};
\node[above=3pt] at (1*\dx,0) {$\widetilde{X}_{h_{i}}$};
\node[above=3pt] at (2*\dx,0) {$U_{h_{i}}^*$};

\node[above=3pt] at (2*\dx,-2) {$b$};
\node[above=3pt] at (8*\dx,-2) {$b^{(1)}$};

\node[above=3pt] at (3*\dx,0) {$U_{h_{i+1}}^*$};
\node[above=3pt] at (4*\dx,0) {$\widetilde{X}_{h_{i+1}}$};
\node[above=3pt] at (5*\dx,0) {$U_{h_{i+1}}^*$};

\node[above=3pt] at (6*\dx,0) {$U_{h_{i+2}}^*$};
\node[above=3pt] at (7*\dx,0) {$\widetilde{X}_{h_{i+2}}$};
\node[above=3pt] at (8*\dx,0) {$U_{h_{i+2}}^*$};

\draw[thick] (2*\dx,0) -- ++(0,-\drop);
\draw[thick] (8*\dx,0) -- ++(0,-\drop);
\draw[thick] ($(2*\dx,-\drop)+(-1,0)$) -- ($(8*\dx,-\drop)+(1,0)$);

\draw[red,dashed] ({-0.5*\dx},1) -- ({-0.5*\dx},-0.8);
\draw[red,dashed] ({8.5*\dx},1) -- ({8.5*\dx},-0.8);
\draw[red,dashed] ({2.45*\dx},1) -- ({2.45*\dx},-0.8);
\draw[red,dashed] ({5.45*\dx},1) -- ({5.45*\dx},-0.8);
```

\end{tikzpicture}
\]
It is straightforward to verify that for every partition
\(
\widehat{\pi}\in \mathcal{NC}\big(B^{(i,i+2)}\big)
\)
satisfying
\(
\widehat{\pi}\vee\widehat{\sigma}=1_5\), where \( \widehat{\sigma}=\{\{1,2,3\}, \{4,5\}\},
\)
one has
\[
\kappa_{\widehat\pi}\Big[U_{h_{i+1}}^*,
\widetilde{X}_{h_{i+1}},
U_{h_{i+1}}^*,
U_{h_{i+2}}^*,
\widetilde{X}_{h_{i+2}}\Big]=0.
\]
Assume now that for every $k\le l$ and every partition
\(
\widehat{\pi}\in \mathcal{NC}\big(B^{(i,i+k)}\big),
\)
the following holds:
\[
\kappa_{\widehat\pi}\Big[U_{h_{i+1}}^*,
\widetilde{X}_{h_{i+1}},
U_{h_{i+1}}^*,...,U_{h_{i+k-1}}^*,
\widetilde{X}_{h_{i+k-1}},U_{h_{i+k-1}}^*,
U_{h_{i+k}}^*,
\widetilde{X}_{h_{i+k}}\Big]=0.
\]
We show that the result remains valid for $l+1$. Thus, we are in the situation illustrated below:

\[
\begin{tikzpicture}[scale=0.8]
    \def\dx{1.6}
    \def\drop{1}

    \foreach \k in {0,1,2,3,5,6,7,8,9}
        \fill (\k*\dx,0) circle (1.5pt);

    \node[above=3pt] at (0*\dx,-2) {$b$};
    \node[above=3pt] at (9*\dx,-2) {$b^{(1)}$};
    
    \node[above=3pt] at (0*\dx,0) {$U_{h_i}^*$};

    \node[above=3pt] at (1*\dx,0) {$U_{h_{i+1}}^*$};
    \node[above=3pt] at (2*\dx,0) {$\widetilde{X}_{h_{i+1}}$};
    \node[above=3pt] at (3*\dx,0) {$U_{h_{i+1}}^*$};

    \node at (4*\dx,0.15) {$\cdots$};
    \node[above=3pt] at (5*\dx,0) {$\widetilde{X}_{h_{i+l}}$};
    \node[above=3pt] at (6*\dx,0) {$U_{h_{i+l}}^*$};
    \node[above=3pt] at (7*\dx,0) {$U_{h_{i+l+1}}^*$};
    \node[above=3pt] at (8*\dx,0) {$\widetilde{X}_{h_{i+l+1}}$};
    \node[above=3pt] at (9*\dx,0) {$U_{h_{i+l+1}}^*$};

    \draw[thick] (0*\dx,0) -- ++(0,-\drop);
    \draw[thick] (9*\dx,0) -- ++(0,-\drop);
    \draw[thick] ($(0*\dx,-\drop)+(-1,0)$) -- ($(9*\dx,-\drop)+(1,0)$);

    \draw[red,dashed] ({0.5*\dx},1) -- ({0.5*\dx},-0.8);
    \draw[red,dashed] ({3.5*\dx},1) -- ({3.5*\dx},-0.8);
    \draw[red,dashed] ({6.5*\dx},1) -- ({6.5*\dx},-0.8);
    \draw[red,dashed] ({9.5*\dx},1) -- ({9.5*\dx},-0.8);

\end{tikzpicture}
\]
Observe that if the point $3(i+1)$ were connected to the point $3(i+1)-2$, then the point $3(i+1)-1$ would necessarily be a singleton and by Corollary \ref{cor_X}, this would imply
\(
\kappa_1\big(\widetilde{X}_{h_{i+1}}\big)=0,
\)

Therefore, the next element to the right of the point $3(i+1)$ belonging to the same block can either be a point of the form $3k$ for some $k>i+1$, or it can be precisely the point $3(i+2)-2$. Suppose first that it is the point $3k$, then we are in a configuration of the form:
\[
\begin{tikzpicture}[scale=1]
    \def\dx{1.5}
    \def\drop{1}

    \foreach \k in {0,1,2,3,5,6,7,9}
        \fill (\k*\dx,0) circle (1.5pt);

    \node[above=3pt] at (0*\dx,-1.8) {$b$};
    \node[above=3pt] at (9*\dx,-1.8) {$b^{(1)}$};
    
    \node[above=3pt] at (0*\dx,0) {$U_{h_{i}}^*$};
    \node[above=3pt] at (1*\dx,0) {$U_{h_{i+1}}^*$};
    \node[above=3pt] at (2*\dx,0) {$\widetilde{X}_{h_{i+1}}$};
    \node[above=3pt] at (3*\dx,0) {$U_{h_{i+1}}^*$};

    \node at (4*\dx,0.1) {$\cdots$};

    \node[above=3pt] at (5*\dx,0) {$U_{h_{k}}^*$};
    \node[above=3pt] at (6*\dx,0) {$\widetilde{X}_{h_{k}}$};
    \node[above=3pt] at (7*\dx,0) {$U_{h_{k}}^*$};

    \node at (8*\dx,0.1) {$\cdots$};

    \node[above=3pt] at (9*\dx,0) {$U_{h_{i+l+1}}^*$};

    \draw[thick] (0*\dx,0) -- ++(0,-\drop);
    \draw[thick] (9*\dx,0) -- ++(0,-\drop);
    \draw[thick] (0*\dx,-\drop) -- (9*\dx,-\drop);
    \draw[thick] (9*\dx,-\drop) -- ++(1,0); 

    \draw[thick] (3*\dx,0) -- ++(0,-0.7);
    \draw[thick] (7*\dx,0) -- ++(0,-0.7);
    \draw[thick] (3*\dx,-0.7) -- (7*\dx,-0.7);
    \draw[thick] (7*\dx,-0.7) -- ++(1,0); 

    \draw[red,dashed] ({0.5*\dx},1) -- ({0.5*\dx},-0.8);
    \draw[red,dashed] ({3.5*\dx},1) -- ({3.5*\dx},-0.8);
    \draw[red,dashed] ({4.5*\dx},1) -- ({4.5*\dx},-0.8);
    \draw[red,dashed] ({7.5*\dx},1) -- ({7.5*\dx},-0.8);
\end{tikzpicture}
\]
Applying the induction hypothesis to the interval partition associated with $\widehat\sigma\in NC\big(B^{(i+1,k)}\big)$ 
we obtain
\[
\kappa_{\widehat\sigma}
\big[
U_{h_{i+2}}^\ast,
\widetilde X_{h_{i+2}},
U_{h_{i+2}}^\ast,
\dots,
U_{h_k}^\ast,
\widetilde X_{h_k}
\big]
=0.
\]
This contradiction shows that the point $3(i+1)$ must be connected to the point $3(i+2)-2$. Repeating the same argument recursively yields
\(
3s \sim_{\widehat{\pi}} 3(s+1)-2
\), for all \(s=i+1,i+2,\ldots,i+l.
\)

By Lemma \ref{lema-2-dem}, every $X$-block contained in
\(
\widehat{\pi}\in NC\big(B^{(i,i+l+1)}\big)
\) has consecutive elements of the form
\(
(\widetilde X_h,\widetilde X_{h^{-1}}).
\)
Therefore, considering the sequence
\(
U_{h_{i+1}}^\ast,
U_{h_{i+2}}^\ast,
U_{h_{i+3}}^\ast,
U_{h_{i+4}}^\ast,
\dots,
U_{h_{i+l}}^\ast,
U_{h_{i+l+1}}^\ast,
\)
and multiplying the corresponding indices according to the nesting structure of the restriction of $\widehat{\pi}$ to the $U$-elements, we obtain
\(
h_{i+l+1}=e,
\)
which is a contradiction. This completes the inductive argument and hence the proof of the lemma. 
\end{proof}
Consequently, we conclude that
\(
b^{(1)}=3(i+1)-2.
\)
That is,
\(
3i\sim_{\pi} 3(i+1)-2\), for all \(i=1,2,\dots,r
\quad (\mathrm{mod}\ r).
\)
Having established this, and using Lemma \ref{lema-2-dem}, it follows immediately that $r$ must be even and that
\[
h_1=h_3=h_5=\cdots=h_{r-1},
\qquad
h_2=h_4=h_6=\cdots=h_r,
\]
with
\(
h_1=h_2^{-1}.
\)
However, this contradicts our assumption that the indices
\(
h_1,\dots,h_r
\)
do not all belong to the same equivalence class of the partition
\[
\{e\},
\qquad
\{g,g^{-1}\},
\quad g\in G\setminus\{e\}.
\]
Therefore,
\(
\kappa_r\bigl(Y_{h_1},Y_{h_2},\ldots,Y_{h_r}\bigr)=0.
\)

\textbf{Case (2):} Suppose that for some $h_1,\ldots,h_r\in G$ we have $h_1=h_2=\cdots=h_r=e$, and for any other $j$ we have $h_j\neq e$. Define
\begin{equation*}\Lambda_j=\big|{m: 1\leq m<j,\ \ h_m=e}\big|,\qquad
\Lambda=\big|{m: 1\leq m\leq r,\ \ h_m=e}\big|.
\end{equation*}
We define the partition $\sigma=\{V_1,\ldots,V_r\}$ as an interval partition of the form
\(
V_j=\{3j-2\Lambda_j-2\}
\)
if $h_j=e$, and
\(
V_j=\{3j-2\Lambda_j-2\,,3j-2\Lambda_j-1\,,3j-2\Lambda_j\}
\)
if $h_j\neq e$. Then, by the product formula for cumulants, we have
\begin{align*}
\kappa_r(Y_{h_{1}},Y_{h_{2}},\ldots,Y_{h_{r}})
=
\sum_{\substack{\pi\in \mathcal{NC}(3r-2\Lambda)\\ \pi\vee\sigma=1_{3r-2\Lambda}}}
\kappa_\pi[Z_1,\ldots,Z_{3r-2\Lambda}],
\end{align*}
where the variables $Z_m$, for $m=1,2,\ldots,3r-2\Lambda$, are defined as follows:
\begin{itemize}
\item If $m\in V_j$ and $h_j=e$, then $Z_m=X_e$.
\item If $m\in V_j$ and $h_j\neq e$, then
\begin{itemize}
\item if $m=3j-2\Lambda_j-2$ or $m=3j-2\Lambda_j$, then $Z_m=U_{h_j}^*$;
\item if $m=3j-2\Lambda_j-1$, then $Z_m=\widetilde{X}_{h_j}$.
\end{itemize}
\end{itemize}
Observe that, by Lemma~\ref{lema ind X y U}, the partitions
\(
\pi\in \mathcal{NC}(3r-2\Lambda)
\)
satisfying $\pi\vee\sigma=1_{3r-2\Lambda}$ must contain blocks that are either of $U$-type or of $X$-type. Otherwise, we would have
\(
\kappa_\pi\big[Z_1,Z_2,\ldots,Z_{3r-2\Lambda}\big]=0.
\)
Let $\pi\in \mathcal{NC}(3r-2\Lambda)$ be such that $\pi\vee\sigma=1_{3r-2\Lambda}$ and suppose that
\begin{equation*}
        \kappa_\pi\big[Z_1,Z_2,...,Z_{3r-2\Lambda}\big]\neq 0.
    \end{equation*}
Let $W\in\pi$ be an $X$-block, and let
\(
b:=3i-2\Lambda_i-2\in W
\)
be such that $Z_b=X_e$. The block $W$ must contain more than one element, since otherwise the condition $\pi\vee\sigma=1_{3r-2\Lambda}$ would not be satisfied.

Without loss of generality, let $b^{(1)}\in W$ be the first element to the right of $b$ that belongs to the same block (the case in which there is only such an element to the left is analogous). By free independence, this element must correspond to a variable of the form $\widetilde X_{h_i}$. At this point, two cases arise: either $Z_{b^{(1)}}=X_e$, or
\(
Z_{b^{(1)}}=U_{h_j}^\ast\widetilde X_{h_j}U_{h_j}^*
\)
with $h_j\neq e$. Observe that if $Z_{b^{(1)}}=X_e$ and $b^{(1)}\neq b+1$, then the condition $\pi\vee\sigma=1_{3r-2s}$ is violated. Indeed, in that case the interval
\(
[b+1,b^{(1)}-1]
\)
is separated from its complement by $\pi$. Therefore, either $b^{(1)}=b+1$ or
\(
Z_{b^{(1)}}=U_{h_j}^\ast\widetilde X_{h_j}U_{h_j}^*
\)
for some $h_j\neq e$.

Assume first that $b^{(1)}$ is such that
\(
Z_{b^{(1)}}=U_{h_j}^\ast\widetilde X_{h_j}U_{h_j}^*
\)
with $h_j\neq e$. We shall prove that
\[
\kappa_\pi\big[Z_1,Z_2,\ldots,Z_{3r-2s}\big]=0.
\]
Consequently, the only partitions $\pi\in \mathcal{NC}(3r-2\Lambda)$ satisfying $\pi\vee\sigma=1_{3r-2\Lambda}$ that may contribute are those for which
\(
b^{(1)}=b+1.
\)
The proof proceeds by induction on the nesting bocks. In the case where the elements $b$ and $b^{(1)}$ have exactly one nesting block between them, that nesting block corresponds to the singleton block containing $U_{h_j}^\ast$, as illustrated in the following diagram:
\[
\begin{tikzpicture}[scale=0.8]
    \def\dx{1.5}
    \def\drop{0.8}
    \foreach \k in {0,...,3}
        \fill (\k*\dx,0) circle (1.5pt);
    \node[above=3pt] at (0*\dx,0) {$X_{e}$};
    \node[above=3pt] at (1*\dx,0) {$U^\ast_{h_j}$};
    \node[above=3pt] at (2*\dx,0) {$\widetilde X_{h_j}$};
    \node[above=3pt] at (3*\dx,0) {$U_{h_j}^\ast$};

    \node[above=3pt] at (0*\dx,-1.6) {$b$};
\node[above=3pt] at (2*\dx,-1.6) {$b^{(1)}$};
    
    \draw[thick] (0*\dx,0) -- ++(0,-\drop);
    \draw[thick] ($(0*\dx,-\drop)+(-1,0)$) -- ($(3*\dx,-\drop)+(-0.5,0)$);
    
    \draw[thick] (1*\dx,0) -- ++(0,-0.5);
    \draw[thick] (2*\dx,0) -- ++(0,-\drop);

     \draw[red,dashed] ({0.5*\dx},1) -- ({0.5*\dx},-0.8);
     \draw[red,dashed] ({-0.5*\dx},1) -- ({-0.5*\dx},-0.8);
     \draw[red,dashed] ({3.5*\dx},1) -- ({3.5*\dx},-0.8);
\end{tikzpicture}
\]
With the above observation and since $h_j\neq e$, it follows that
\begin{equation*}
\kappa_\pi\big[Z_1,Z_2,\ldots,Z_{3r-2\Lambda}\big]
=
\kappa_{\pi\backslash{U_{h_j}^\ast}}\big[Z_1,Z_2,\ldots,Z_{3r-2\Lambda}\big]
\kappa_1\big(U_{h_j}^\ast\big)
=0.
\end{equation*}
Assume that whenever there are between $1$ and $l$ nested blocks between the elements
\(
b=3i-2\Lambda_i-2\), and \(
b^{(1)}=3j-2\Lambda_j-1,
\)
we have
\begin{equation*}
\kappa_\pi\big[Z_1,Z_2,\ldots,Z_{3r-2\Lambda}\big]=0.
\end{equation*}
To prove the case in which there are $l+1$ nestings blocks, observe that there are two possibilities. Case (2.i): There exists an element $X_e$ between $Z_b$ and $Z_{b^{(1)}}$. Case (2.ii): There is no element $X_e$ between $Z_b$ and $Z_{b^{(1)}}$. In both cases we are going to prove that $b^{(1)}=b+1.$
\begin{itemize}[leftmargin=*]
    \item \textbf{Case (2.i)} Suppose there exists an element \(X_e\) between \(Z_b\) and \(Z_{b^{(1)}}\), which we denote by \(Z_{b^{(2)}} := X_e\), and necessarily satisfies
\(
b < b^{(2)} < b^{(1)}.
\)
The element \(b^{(2)}\) cannot be a singleton; otherwise, this would contradict the condition \(\pi \vee \sigma = 1_{3r-2s}\). Without loss of generality, let \(b^{(3)}\) denote the next element to the right of \(b^{(2)}\) belonging to the same block. This element may either satisfy \(b^{(3)} = b^{(2)} + 1\) or \(b^{(3)} > b^{(2)} + 1\).

Assume that \(b^{(3)} > b^{(2)} + 1\). Then it must hold that \(Z_{b^{(3)}} \neq X_e\), since otherwise we would again contradict the condition \(\pi \vee \sigma = 1_{3r-2s}\). Therefore, applying the induction hypothesis to the elements lying between \(b^{(2)}\) and \(b^{(3)}\), we conclude that
\(
\kappa_\pi\big[Z_1,Z_2,\ldots,Z_{3r-2}\big]=0.
\)
This proves that \(b^{(3)} = b^{(2)} + 1\), and hence
\(
Z_{b^{(3)}} = Z_{b^{(2)}+1} = X_e.
\)

The elements \(b^{(2)}\) and \(b^{(3)}\) cannot be the only elements in their block, for otherwise the condition \(\pi \vee \sigma = 1_{3r-2s}\) would be violated. Thus, there exists another element in the same block. Applying the same argument as before, we conclude that this new element \(b^{(4)}\) satisfies
\(
Z_{b^{(4)}} = Z_{b^{(2)}+2} = X_e.
\)
Repeating this argument inductively, we conclude that between \(b\) and \(b^{(1)}\) there is an \(X\)-block given by
\(
\{b,b+1,b+2,\ldots,b^{(1)}-1\},
\)
corresponding to the elements
\(
\{X_e,X_e,X_e,\ldots,X_e\}.
\)
Thus, the situation reduces exactly to the one illustrated in the following diagram.

\[
\begin{tikzpicture}[scale=0.8]
    \def\dx{1.5}
    \def\drop{0.8}
    \def\bigdrop{1.3}

    \foreach \k in {0,1,2,4,5,6,7,8}
        \fill (\k*\dx,0) circle (1.5pt);

    \node[above=3pt] at (0*\dx,0) {$X_{e}$};

    \node[above=3pt] at (1*\dx,0) {$X_e$};
    \node[above=3pt] at (2*\dx,0) {$X_e$};

    \node at (3*\dx,0.15) {$\cdots$};

    \node[above=3pt] at (4*\dx,0) {$X_e$};
    \node[above=3pt] at (5*\dx,0) {$X_e$};

    \node[above=3pt] at (6*\dx,0) {$U_{h_j}^\ast$};
    \node[above=3pt] at (7*\dx,0) {$\widetilde X_{h_j}$};
    \node[above=3pt] at (8*\dx,0) {$U_{h_j}^\ast$};

    \foreach \k in {1,2,4,5}
        \draw[thick] (\k*\dx,0) -- ++(0,-\drop);
    \draw[thick] (1*\dx,-\drop) -- (5*\dx,-\drop);

    \draw[thick] (0*\dx,0) -- ++(0,-\bigdrop);
    \draw[thick] (7*\dx,0) -- ++(0,-\bigdrop);
    \draw[thick] ($(0*\dx,-\bigdrop)+(-1,0)$) -- ($(8*\dx,-\bigdrop)+(-0.5,0)$);
    \draw[thick] (6*\dx,0) -- ++(0,-\drop);

    \draw[red,dashed] ({-0.5*\dx},1.3) -- ({-0.5*\dx},-1.3);
    \draw[red,dashed] ({1.5*\dx},1.3) -- ({1.5*\dx},-1.3);
\draw[red,dashed] ({0.5*\dx},1.3) -- ({0.5*\dx},-1.3);
\draw[red,dashed] ({2.5*\dx},1.3) -- ({2.5*\dx},-1.3);
\draw[red,dashed] ({4.5*\dx},1.3) -- ({4.5*\dx},-1.3);
\draw[red,dashed] ({5.5*\dx},1.3) -- ({5.5*\dx},-1.3);
\draw[red,dashed] ({3.5*\dx},1.3) -- ({3.5*\dx},-1.3);
\draw[red,dashed] ({8.5*\dx},1.3) -- ({8.5*\dx},-1.3);

\end{tikzpicture}
\]
Thus, since \(h_j \neq e\), we have
\(
\kappa_1(U_{h_j}^{\ast}) = 0,
\)
and therefore
\(
\kappa_\pi\big[Z_1,Z_2,\ldots,Z_{3r-2s}\big]=0.
\)
This proves that \(b^{(1)} = b+1\).

    \item \textbf{Case (2.ii)} Suppose that there is no element \(X_e\) between \(Z_b\) and \(Z_{b^{(1)}}\). Observe that the element \(3(i+1)-2s_{i+1}\) must be connected to some element \(3k-2s_k\); otherwise, this would contradict the condition \(\pi\vee\sigma=1_{3r-2s}\).
\[
\begin{tikzpicture}[scale=0.65]
    \def\dx{1.5}
    \def\drop{1}

    \foreach \k in {0,1,2,3,5,6,7,9,10,11,12,13}
        \fill (\k*\dx,0) circle (1.5pt);

    \node[above=3pt] at (0*\dx,0) {$X_{e}$};
    \node[above=3pt] at (1*\dx,0) {$U_{h_{i+1}}^*$};
    \node[above=3pt] at (2*\dx,0) {$\widetilde{X}_{h_{i+1}}$};
    \node[above=3pt] at (3*\dx,0) {$U_{h_{i+1}}^*$};

    \node at (4*\dx,0.1) {$\cdots$};

    \node[above=3pt] at (5*\dx,0) {$U_{h_{k}}^*$};
    \node[above=3pt] at (6*\dx,0) {$\widetilde{X}_{h_{k}}$};
    \node[above=3pt] at (7*\dx,0) {$U_{h_{k}}^*$};

    \node at (8*\dx,0.1) {$\cdots$};
    \node[above=3pt] at (9*\dx,0) {$U_{h_{j-1}}^*$};
    \node[above=3pt] at (10*\dx,0) {$\widetilde{X}_{h_{j-1}}$};
    \node[above=3pt] at (11*\dx,0) {$U_{h_{j-1}}^*$};
\node[above=3pt] at (12*\dx,0) {$U_{h_{j}}^*$};

    \node[above=3pt] at (13*\dx,0) {$\widetilde X_{h_j}$};

    \draw[thick] (0*\dx,0) -- ++(0,-\drop);
    \draw[thick] (13*\dx,0) -- ++(0,-\drop);
    \draw[thick] (0*\dx,-\drop) -- (13*\dx,-\drop);
    \draw[thick] (13*\dx,-\drop) -- ++(1,0); 

    \draw[thick] (3*\dx,0) -- ++(0,-0.7);
    \draw[thick] (7*\dx,0) -- ++(0,-0.7);
    \draw[thick] (3*\dx,-0.7) -- (7*\dx,-0.7);
    \draw[thick] (7*\dx,-0.7) -- ++(1,0); 
\draw[red,dashed] ({-0.5*\dx},1.3) -- ({-0.5*\dx},-0.8);
    \draw[red,dashed] ({0.5*\dx},1.3) -- ({0.5*\dx},-0.8);
    \draw[red,dashed] ({3.5*\dx},1.3) -- ({3.5*\dx},-0.8);
    \draw[red,dashed] ({4.5*\dx},1.3) -- ({4.5*\dx},-0.8);
    \draw[red,dashed] ({7.5*\dx},1.3) -- ({7.5*\dx},-0.8);
    \draw[red,dashed] ({8.5*\dx},1.3) -- ({8.5*\dx},-0.8);
    \draw[red,dashed] ({11.5*\dx},1.3) -- ({11.5*\dx},-0.8);

\end{tikzpicture}
\]
From this point on, we are in exactly the same situation as in Case 1). By an analogous argument, we can conclude that
\begin{equation*}
    3p-s_p\sim_\pi 3(p+1)-s_{p+1}-2,\qquad \text{for every }p=i+1,i+2,\ldots,j-1.
\end{equation*}
Moreover, by Lemma \ref{lema-2-dem}, each $X$-block consists of consecutive elements of the form $X_h$ and $X_{h^{-1}}$. Finally, consider the sequence of $U$-elements lying between the elements $b$ and $b'$, namely,
$$
U_{h_{i+1}}^\ast,U_{h_{i+1}}^\ast,U_{h_{i+2}}^\ast,U_{h_{i+2}}^\ast,\ldots,U_{h_k}^\ast,U_{h_k}^\ast,\ldots,U_{h_{j-1}}^\ast,U_{h_{j-1}}^\ast,U_{h_j}^\ast.
$$
We multiply the indices according to the nesting order of the partition $\pi$ restricted to the $U$-elements in the above sequence. Carrying out this multiplication yields $h_j=e$, which is a contradiction. Therefore, it follows that $b^{(1)}=b+1$.

\end{itemize}
Since in both cases, Case (2.i) and Case (2.ii), it was proved that $b^{(1)}=b+1$, and hence $Z_{b^{(1)}}=X_e$. The elements $b$ and $b^{(1)}$ cannot be the only elements in their block, for otherwise this would contradict the condition $\pi\vee\sigma=1_{3r-2s}$. Therefore, there exists another element in the same block, and by an inductive argument we conclude that $\pi$ contains only one block of the form
$$
\big\{3(1)-2s_1-2,\,3(2)-2s_2-2,\,3(3)-2s_3-2,\ldots,3(r)-2s_r-2\big\},
$$
corresponding to the elements $\{X_e,X_e,X_e,\ldots,X_e\}$. However, this contradicts the assumption that the elements $h_1,h_2,\ldots,h_r$ belong to the same equivalence class $\{e\}$ or $\{g,g^{-1}\}$ for some $g\in G\setminus\{e\}$. Therefore, we conclude that
\begin{equation*}
    \kappa_r\big(Y_{h_1},Y_{h_2},\ldots,Y_{h_r}\big)=0.
\end{equation*}

\end{proof}
\begin{remark}
It is important to emphasize that two groups of the same order generally induce different left $G$-circulant decompositions of a given random matrix, For example lets consider $G_1=S_3$, and let $A\in M_{6\times 6}(\mathcal{A})$ be free over $M_{6\times 6}(\mathbb{C})$. Writing
\(
A=(A_{ij})_{1\leq i,j\leq 6},
\)
with $A_{ij}\in\mathcal{A}$, if we take the left $G_1$-circulant decomposition from $A^t$ its given by

\[
\scalebox{0.8}{$
\begin{aligned}
    {A^t}=&\begin{bmatrix}
        A_{1,1} & \cdot & \cdot & \cdot & \cdot & \cdot\\
        \cdot & A_{2,2} & \cdot & \cdot & \cdot & \cdot\\ 
        \cdot & \cdot & A_{3,3} & \cdot & \cdot & \cdot\\ 
        \cdot & \cdot & \cdot & A_{4,4} & \cdot & \cdot\\ 
        \cdot & \cdot & \cdot & \cdot & A_{5,5} & \cdot\\ 
        \cdot & \cdot & \cdot & \cdot & \cdot & A_{6,6}\\ 
    \end{bmatrix}
    +
    \begin{bmatrix}
        \cdot & A_{2,1} & \cdot & \cdot & \cdot & \cdot\\
        A_{1,2} & \cdot & \cdot & \cdot & \cdot & \cdot\\ 
        \cdot & \cdot & \cdot & \cdot & A_{5,3} & \cdot\\ 
        \cdot & \cdot & \cdot & \cdot & \cdot & A_{6,4}\\ 
        \cdot & \cdot & A_{3,5} & \cdot & \cdot & \cdot\\ 
        \cdot & \cdot & \cdot & A_{4,6} & \cdot & \cdot\\ 
    \end{bmatrix}
    +
    \begin{bmatrix}
\cdot & \cdot & A_{3,1} & \cdot & \cdot & \cdot \\
\cdot & \cdot & \cdot & \cdot & \cdot & A_{6,2} \\
A_{1,3} & \cdot & \cdot & \cdot & \cdot & \cdot \\
\cdot & \cdot & \cdot & \cdot & A_{5,4} & \cdot \\
\cdot & \cdot & \cdot & A_{4,5} & \cdot & \cdot \\
\cdot & A_{2,6} & \cdot & \cdot & \cdot & \cdot
\end{bmatrix}
\\
&+
\begin{bmatrix}
\cdot & \cdot & \cdot & A_{4,1} & \cdot & \cdot \\
\cdot & \cdot & \cdot & \cdot & A_{5,2} & \cdot \\
\cdot & \cdot & \cdot & \cdot & \cdot & A_{6,3} \\
A_{1,4} & \cdot & \cdot & \cdot & \cdot & \cdot \\
\cdot & A_{2,5} & \cdot & \cdot & \cdot & \cdot \\
\cdot & \cdot & A_{3,6} & \cdot & \cdot & \cdot
\end{bmatrix}
+
\begin{bmatrix}
\cdot & \cdot & \cdot & \cdot & \cdot & A_{6,1} \\
\cdot & \cdot & A_{3,2} & \cdot & \cdot & \cdot \\
\cdot & \cdot & \cdot & A_{4,3} & \cdot & \cdot \\
\cdot & A_{2,4} & \cdot & \cdot & \cdot & \cdot \\
A_{1,5} & \cdot & \cdot & \cdot & \cdot & \cdot \\
\cdot & \cdot & \cdot & \cdot & A_{5,6} & \cdot
\end{bmatrix}
+
\begin{bmatrix}
\cdot & \cdot & \cdot & \cdot & A_{5,1} & \cdot \\
\cdot & \cdot & \cdot & A_{4,2} & \cdot & \cdot \\
\cdot & A_{2,3} & \cdot & \cdot & \cdot & \cdot \\
\cdot & \cdot & A_{3,4} & \cdot & \cdot & \cdot \\
\cdot & \cdot & \cdot & \cdot & \cdot & A_{6,5} \\
A_{1,6} & \cdot & \cdot & \cdot & \cdot & \cdot
\end{bmatrix}\\
=&A_0+A_1+A_2+A_3+A_4+A_5\\
=&A_0+A_1+A_2+A_3+\big(A_4+A_5\big).
\end{aligned}
$}
\]
In the other hand, consider $G_2=\mathbb{Z}_6$, if we take now the left $G_2$-circulant decomposition from $A^t$ which is given by
\[
\scalebox{0.8}{$
\begin{aligned}
     A^t=&\begin{bmatrix}
        A_{1,1} & \cdot & \cdot & \cdot & \cdot & \cdot\\
        \cdot & A_{2,2} & \cdot & \cdot & \cdot & \cdot\\ 
        \cdot & \cdot & A_{3,3} & \cdot & \cdot & \cdot\\ 
        \cdot & \cdot & \cdot & A_{4,4} & \cdot & \cdot\\ 
        \cdot & \cdot & \cdot & \cdot & A_{5,5} & \cdot\\ 
        \cdot & \cdot & \cdot & \cdot & \cdot & A_{6,6}\\ 
\end{bmatrix}
+
\begin{bmatrix}
        \cdot & A_{2,1} & \cdot & \cdot & \cdot & \cdot\\
        \cdot & \cdot & A_{3,2} & \cdot & \cdot & \cdot\\ 
        \cdot & \cdot & \cdot & A_{4,3} & \cdot & \cdot\\ 
        \cdot & \cdot & \cdot & \cdot & A_{5,4} & \cdot\\ 
        \cdot & \cdot & \cdot & \cdot & \cdot & A_{6,5}\\ 
        A_{1,6} & \cdot & \cdot & \cdot & \cdot & \cdot\\ 
\end{bmatrix}
+
\begin{bmatrix}
\cdot & \cdot & A_{3,1} & \cdot & \cdot & \cdot \\
\cdot & \cdot & \cdot & A_{4,2} & \cdot & \cdot \\
\cdot & \cdot & \cdot & \cdot & A_{5,3} & \cdot \\
\cdot & \cdot & \cdot & \cdot & \cdot & A_{6,4} \\
A_{5,1} & \cdot & \cdot & \cdot & \cdot & \cdot \\
\cdot & A_{2,6} & \cdot & \cdot & \cdot & \cdot
\end{bmatrix}
\\
&+
\begin{bmatrix}
\cdot & \cdot & \cdot & A_{4,1} & \cdot & \cdot \\
\cdot & \cdot & \cdot & \cdot & A_{5,2} & \cdot \\
\cdot & \cdot & \cdot & \cdot & \cdot & A_{6,3} \\
A_{1,4} & \cdot & \cdot & \cdot & \cdot & \cdot \\
\cdot & A_{2,5} & \cdot & \cdot & \cdot & \cdot \\
\cdot & \cdot & A_{3,6} & \cdot & \cdot & \cdot
\end{bmatrix}
+
\begin{bmatrix}
\cdot & \cdot & \cdot & \cdot & A_{5,1} & \cdot \\
\cdot & \cdot & \cdot & \cdot & \cdot & A_{6,2} \\
A_{1,3} & \cdot & \cdot & \cdot & \cdot & \cdot \\
\cdot & A_{2,4} & \cdot & \cdot & \cdot & \cdot \\
\cdot & \cdot & A_{3,5} & \cdot & \cdot & \cdot \\
\cdot & \cdot & \cdot & A_{4,6} & \cdot & \cdot
\end{bmatrix}
+
\begin{bmatrix}
\cdot & \cdot & \cdot & \cdot & \cdot & A_{6,1} \\
A_{1,2} & \cdot & \cdot & \cdot & \cdot & \cdot \\
\cdot & A_{2,3} & \cdot & \cdot & \cdot & \cdot \\
\cdot & \cdot & A_{3,4} & \cdot & \cdot & \cdot \\
\cdot & \cdot & \cdot & A_{4,5} & \cdot & \cdot \\
\cdot & \cdot & \cdot & \cdot & A_{5,6} & \cdot
\end{bmatrix}\\ 
=&B_0+B_1+B_2+B_3+B_4+B_5\\
=&B_0+\big(B_1+B_5\big)+\big( B_2+B_4\big)+B_3.
\end{aligned}
$}
\]
By the Theorem \ref{main-thm} we conclude that $\big\{B_0,\ \big(B_1+B_5\big),\ \big( B_2+B_4\big), \ B_3\big\}$ and $\big\{A_0, \ A_1,\ A_2,\ A_3,\ \big(A_4+A_5\big)\big\}$ are freely independent families of random matrices.
\end{remark}\label{main-remark}
In general, let \(G\) be a finite group, let \(A\in M_{|G|}(\mathcal{A})\) be free from \(M_{|G|}(\mathbb{C})\), and consider the left \(G\)-circulant decomposition of \(A^t\),
\[
A^t=\sum_{g\in G}Y_g,
\]
where
\(
Y_g=U_g^*\odot A^t.
\) Let
\(
S:=\{g\in G:\ g=g^{-1}\neq e\}
\)
denote the set of nontrivial involutions of \(G\), and let \(R\subseteq G\) be a complete set of representatives of the inverse pairs
\(\{g,g^{-1}\}\) with \(g\neq g^{-1}\). Then
\begin{equation}
    \label{left-G-circ-decomposition-free}
A^t
=
Y_e
+\sum_{g\in S}Y_g
+\sum_{g\in R}\left(Y_g+Y_{g^{-1}}\right),
\end{equation}
and, by Theorem~\ref{main-thm}, the family of summands is free. Consequently, the asymptotic eigenvalue distribution of \(A^t\) is given by
\[
\mu_{A^t}
=
\mu_{Y_e}
\mathop{\mathlarger{\mathlarger{\boxplus}}}\limits_{g\in S}
\mu_{Y_g}
\mathop{\mathlarger{\mathlarger{\boxplus}}}\limits_{g\in R}
\mu_{Y_g+Y_{g^{-1}}}.
\]
Therefore, the limiting spectral distribution of the partial transpose is determined by the distributions of the identity component, $\mu_{Y_e}$, the components indexed by nontrivial involutions, $\mu_{Y_g}$ and the sums corresponding to inverse pairs, $\mu_{Y_g+Y_{g^{-1}}}$.

\subsection{Distributions from $Y_e$, $Y_g$ for $g\in S$ and $Y_g+Y_{g^{-1}}$ for $g\in R$}
In this subsection, we will study the distributions of $Y_e$, $Y_g$ for $g\in S$ and $Y_g+Y_{g^{-1}}$ for $g\in R$ via free cumulants and moments.

Recall that the limiting spectral distribution of $X\in M_{|G|}(\mathcal{A})$ is denoted by $\mu$. We will show that the distributions of these matrices can be expressed in terms of $\mu$, as stated in the following propositions.

\begin{prop}\label{prop: Y_e}
The limiting spectral distribution from $Y_e$ is $\displaystyle\mu_{Y_e}=\big(D_{\frac{1}{|G|}}\mu\big)^{\boxplus |G|}$.
\end{prop}
\begin{proof}
Recall that $Y_e=X_e=\sum_{h\in G}E_{h,h}XE_{h,h}$. Since $E_{h,g}E_{h^\prime,g^\prime}=E_{h,g^\prime}\delta_{g,h^\prime}$, for every $n\ge1$, it is straightforward to see that
\begin{align*}
    Y_e^{m}=\sum_{h\in G}E_{h,h}XE_{h,h}XE_{h,h}\cdots E_{h,h}XE_{h,h}.
\end{align*}
Now, we apply $\varphi$ to $Y_e^m$ and use that $\varphi$ is tracial, then 
\begin{align*}
    \varphi(Y_e^m)=&\sum_{h\in G}\varphi(E_{h,h}XE_{h,h}XE_{h,h}\cdots E_{h,h}XE_{h,h})=\sum_{h\in G}\varphi(XE_{h,h}XE_{h,h}\cdots E_{h,h}XE_{h,h}),
\end{align*}
since the matrix $X$ is free independent from $M_{|G|}(\mathbb C)$, Theorem \ref{thm: as ind bs} implies
   \begin{align*} \varphi(Y_e^m)=\sum_{h\in G}\sum_{\pi\in \mathcal{NC}(m)}\kappa_\pi[X,\ldots ,X]\varphi_{Kr(\pi)}[E_{h,h},E_{h,h},\ldots,E_{h,h}]
    =&\sum_{h\in G}\sum_{\pi\in \mathcal{NC}(m)}\frac{1}{|G|^{m-|\pi|+1}}\kappa_\pi[X,\ldots ,X]\\
=&|G|\sum_{\pi\in \mathcal{NC}(m)}\frac{1}{|G|^{m-|\pi|+1}}\kappa_\pi[X,\ldots ,X]\\
=& \sum_{\pi\in \mathcal{NC}(m)}\frac{|G|^{|\pi|}}{|G|^{m}}\kappa_\pi[X,\ldots ,X]
\\=& \sum_{\pi\in \mathcal{NC}(m)}|G|^{|\pi|}\kappa_\pi\left[\frac{X}{|G|},\ldots ,\frac{X}{|G|}\right]
\end{align*} 
which proves that the cumulants of $Y_e$ are given by $\kappa_n(Y_e)=|G|\kappa_n(X/|G|)$, and thus the 
limiting spectral distribution from $Y_e$ is $\displaystyle\mu_{Y_e}=\big(D_{\frac{1}{|G|}}\mu\big)^{\boxplus |G|}$.
\end{proof}
To complete the following two propositions in this subsection, we will often use partitions whose blocks have even size. We call these partitions \emph{even}. We will also consider the notation of a partition $\pi$ couples in a cyclic way (c.c.w.), introduced in Lecture 14 from \cite{nica2006lectures}:

\noindent\textbf{Notation}: Let a partition $\pi\in NC(n)$ and an $n$-tuple of double-indices $\big(i(1)j(1),i(2)j(2),\ldots,i(n)j(n)\big)$ be given. Then we say that $\pi$ \emph{couples in a cyclic way} (c.c.w., for short) the indices $\big(i(1)j(1),\ldots,i(n)j(n)\big)$ if we have for each block $\{r_1<r_2<\cdots<r_s\}\in\pi$ that $j(r_k)=i(r_{k+1})$ for all $k=1,\ldots,s$ (where we put $r_{s+1}:=r_1$).
\begin{prop}\label{prop: Y_g}
 For every $g\in S$, $\displaystyle\mu_{Y_g}=\big(D_{\frac{1}{|G|}}\mu\big)^{\boxplus\frac{|G|}{2}}\boxplus\big(D_{-\frac{1}{|G|}}\mu\big)^{\boxplus\frac{|G|}{2}} $.
\end{prop}
\begin{proof}
Let $m\geq 1$, by the moment-cumulant formula we have that
\begin{equation*}
    \varphi\big((Y_g)^m)=\sum_{\pi\in \mathcal{NC}(m)}\kappa_{\pi}[Y_g,\ldots,Y_g]=\sum_{\pi\in \mathcal{NC}(m)}\prod_{V\in\pi}\kappa_{V}[Y_g,\ldots,Y_g],
\end{equation*}
observe that in the Theorem \ref{main-thm} we proved that $\{Y_g,Y_{g^{-1}}\}$ is an $R$-diagonal pair, it implies that 
\begin{equation*}
    \kappa_{|V|}[Y_g,\ldots,Y_g]=0,
\end{equation*}
unless $|V|$ is even and we have alternated elements $\kappa_{|V|}[Y_g,Y_{g^{-1}},\ldots ,Y_g,Y_{g^{-1}}]$, the alternated elements property is always true because we are assuming $g\in S$. Thus, all the odd moments $\varphi\big((Y_g)^{2m+1}\big)$ are equal to $0$. Then, the only non possible zero moments are 
\begin{equation}
    \label{ec-9}\varphi\big((Y_g)^{2m})=\sum_{\substack{\pi\in\mathcal{NC}(2m)\\ \pi\text{ is even }\\}}\kappa_{\pi}[Y_g,\ldots,Y_g].
\end{equation}
On the other hand, notice that, since $g=g^{-1}$, then
$U^*_g=U^*_{g^{-1}}=(U^*_{g})^{-1}$, and $X_g=X_{g^{-1}}$. This implies that $Y_g=U_g^\ast X_{g^{-1}}U_g^\ast=U_g^\ast X_{g}(U_g^\ast)^{-1}$ and hence, by the traciality of $\varphi$, we get
    \begin{equation*}
        \varphi\big((Y_g)^{2m}\big)=\varphi\big((U_g^\ast X_{g}(U_g^\ast)^{-1})^{2m}\big)=\varphi\big((U_g^\ast)^{2m}X_g^{2m}((U_g^\ast)^{-1})^{2m}\big)=\varphi\big((X_g)^{2m}\big).
    \end{equation*}
    Given that $X_g=\sum_{h\in G}E_{h,h}XE_{gh,gh}$ we have that
    \begin{equation*}
        \big(X_g\big)^{2m}=\sum_{h_1,\ldots, h_{2m}}E_{h_1,h_1}XE_{gh_1,gh_1}E_{h_2,h_2}XE_{gh_2,gh_2}\cdots E_{h_{2m}h_{2m}}XE_{gh_{2m},gh_{2m}}.
    \end{equation*}
    Taking $\varphi$, using cyclicity of the trace to move $E_{h_1,h_1}$ to the end 
    \begin{equation*}
        \varphi\big((Y_g)^{2m}\big)=\sum_{h_1,\ldots, h_{2m}}\varphi\big(XE_{gh_1,gh_1}E_{h_2,h_2}XE_{gh_2,gh_2}\cdots E_{h_{2m}h_{2m}}XE_{gh_{2m},gh_{2m}}E_{h_1,h_1}\big),
    \end{equation*}
    and using that  \begin{equation*}
        E_{gh_k,gh_k}E_{h_{k+1},h_{k+1}}=\delta_{gh_k,h_{k+1}}E_{h_{k+1},h_{k+1}},\qquad \forall \,k=1,2,\ldots,2m,
    \end{equation*}
    we see that the product is nonzero only if $h_{k+1}=gh_k$ for every $k$. Since $g^2=e$, this leaves exactly one parameter, $h_1\in G$, with $h_2=gh_1$, $h_3=g^2h_1=h_1$, and so on, alternating $h_1,h_2,h_1,h_2,\ldots$, with the condition $h_1=gh_{2m}$. Consequently
    \begin{equation*}
        \varphi\big((Y_g)^{2m}\big)=\sum_{h_1\in G}\varphi\big(XE_{h_2,h_2}XE_{h_1,h_1}XE_{h_2,h_2}XE_{h_1,h_1}\cdots XE_{h_2,h_2}XE_{h_1,h_1}\big),\qquad h_2=gh_1.
    \end{equation*}
Since $X$ is free from $M_{|G|}(\mathbb C)$, Theorem \ref{thm: as ind bs} implies that
    \begin{align*}
        \varphi \big((Y_g)^{2m}\big)=\sum_{h_1}\sum_{\pi\in\mathcal{NC}(2m)}\kappa_{\pi}[X,X,\ldots,X]\varphi_{Kr(\pi)}[E_{h_2,h_2},E_{h_1,h_1},E_{h_2,h_2},E_{h_1,h_1},\ldots ,E_{h_{2}h_{2}},E_{h_1,h_1}].
    \end{align*}
    To observe the factor $\varphi_{Kr(\pi)}$ we identify the $2m$-tuple of double indices
    \begin{equation*}
        (\vec{i},\vec j):=\big(i(1)j(1),i(2)j(2),\ldots,i(2m-1)j(2m-1),i(2m)j(2m)\big)=\big(h_2h_2,h_1h_1,\ldots,h_2h_2,h_1h_1\big),
    \end{equation*}
    that is
    \begin{equation*}
        i(2k-1)=j(2k-1)=h_2,\quad i(2k)=j(2k)=h_1,\qquad \forall\, k=1,2,\ldots,2m
    \end{equation*}
    Then
    \begin{align*}
        \varphi_{Kr(\pi)}\big[E_{h_2,h_2},E_{h_1,h_1},\cdots ,E_{h_{2}h_{2}},E_{h_1,h_1}\big]=&\varphi_{Kr(\pi)}\big[E_{i(1)j(1)},E_{i(2)j(2)},\cdots ,E_{i(2m-1)j(2m-1)},E_{i(2m)j(2m)}\big]\\ 
        =&\prod_{\underset{B=\{b_1<\cdots <b_{|B|}\}}{B\in Kr(\pi)}}\varphi\big(E_{i(b_1),j(b_1)}E_{i(b_2),j(b_2)}\cdots E_{i(b_{|B|})j(b_{|B|})}\big)\\
        =&\prod_{\underset{B=\{b_1<\cdots <b_{|B|}\}}{B\in Kr(\pi)}}\frac{1}{|G|}\delta_{\{j(b_1)=i(b_2),\ldots,j(b_{|B|-1})=i(b_{|B|}),j(b_{|B|})=i(b_1)\}}.
    \end{align*}
    It implies that, if $Kr(\pi)$ c.c.w. the indices $(\vec i, \vec j)$,
    \begin{equation*}
        \varphi_{Kr(\pi)}[E_{h_2,h_2},E_{h_1,h_1},E_{h_2,h_2},E_{h_1,h_1}\cdots E_{h_{2}h_{2}},E_{h_1,h_1}]=|G|^{-|Kr(\pi)|},
    \end{equation*}
    and $0$ otherwise. Note that this condition, and its value, do not depend on $h_1$, and using $|Kr(\pi)|+|\pi|=2m+1$, 
    \begin{align*}
        \varphi \big((Y_g)^{2m}\big)=&\sum_{h_1}\sum_{\underset{Kr(\pi)\, c.c.w.\, (\vec i,\vec j)}{\pi\in\mathcal{NC}(2m)}}\kappa_{\pi}[X,\cdots,X]|G|^{-2m-1+|\pi|} 
        =|G|\sum_{\underset{Kr(\pi)\, c.c.w.\, (\vec i,\vec j)}{\pi\in\mathcal{NC}(2m)}}\kappa_{\pi}[X,\cdots,X]|G|^{-2m-1+|\pi|} .
    \end{align*}
    Even more, we have $\kappa_\pi[X,\ldots,X]\,|G|^{-2m}=\kappa_\pi[|G|^{-1}X,\ldots,|G|^{-1}X]$, so 
    \begin{equation*}
    \varphi\big((Y_g)^{2m}\big)=\sum_{\underset{Kr(\pi)\, c.c.w.\, (\vec i,\vec j)}{\pi\in \mathcal{NC}(2m)}}\kappa_\pi\big[|G|^{-1}X,\ldots,|G|^{-1}X\big]\,|G|^{|\pi|}.
    \end{equation*}
    Now, observe that the condition $Kr(\pi)$ c.c.w. $(\vec{i},\vec{j})$ implies that each block of $Kr(\pi)$ contains integers with the same parity modulo $2$. This occurs if and only if every block of $\pi$ has even size (see Proposition 3.1 in \cite{arizmendi2012products}). Therefore,
    \begin{equation*}
        \varphi\big((Y_g)^{2m}\big)=\sum_{\substack{\pi\in\mathcal{NC}(2m)\\ \pi\text{ is even}}}\kappa_\pi\big[|G|^{-1}X,\ldots,|G|^{-1}X\big]\,|G|^{|\pi|}.
    \end{equation*}
    Using the expression above together with \eqref{ec-9}, we obtain that    
    \begin{equation*}
        \kappa_{2m}(Y_g,\ldots,Y_g)=|G|\kappa_{2m}(|G|^{-1}X,\cdots,|G|^{-1}X)=|G|\frac{1}{|G|^{2m}}\kappa_{2m}(X,\ldots,X)=|G|^{1-2m}\kappa_{2m}(X,\ldots,X).
    \end{equation*}
    Finally, for $\nu:=\big(D_{1/|G|}\mu\big)^{\boxplus|G|/2}\boxplus\big(D_{-1/|G|}\mu\big)^{\boxplus|G|/2}$ we have $k_{2m+1}=0$ and
\[
\kappa_{2m}(\nu)=\frac{|G|^{1-{2m}}}2\big(1+(-1)^{2m}\big)\kappa_{2m}(X,\ldots,X)=|G|^{1-2m}\kappa_{2m}(X,\dots,X),
\]
which coincides exactly with the cumulants of $Y_g$ just computed. This proves the statement.
\end{proof}

\begin{prop}\label{prop: Y_g+Y_g-1}
 For every $g\in R$, then $\displaystyle \mu_{Y_g+Y_{g^{-1}}}=\big(D_{\frac{1}{|G|}}\mu\big)^{\boxplus|G|}\boxplus\big(D_{-\frac{1}{|G|}}\mu\big)^{\boxplus|G|}$.
\end{prop}
\begin{proof}
Let $m\geq 1$, similarly from the last proposition, with the observation that in the Theorem \ref{main-thm} we proved that $\{Y_g,Y_g^{-1}\}$ is an $R$-diagonal pair, we have that the odd moments $\varphi\big((Y_g+Y_g^{-1})^{2m+1}\big)$ are equal to zero, and the even moments are
\begin{equation}
    \label{ec-10}\varphi\big((Y_g+Y_g^{-1})^{2m})=\sum_{\substack{\pi\in\mathcal{NC}(2m)\\ \pi\text{ is even }}}\kappa_{\pi}[Y_g+Y_g^{-1},\ldots,Y_g+Y_g^{-1}],
\end{equation}
Now, only the two alternating cumulants in the above sum can be non-zero, so
\begin{equation*}
\kappa_{2m}(Y_g+Y_{g^{-1}},\ldots,Y_g+Y_{g^{-1}})
    =
    \kappa_{2m}(Y_g,Y_{g^{-1}},\ldots,Y_g,Y_{g^{-1}})+
    \kappa_{2m}(Y_{g^{-1}},Y_g,\ldots,Y_{g^{-1}},Y_g).
\end{equation*}
Since $\varphi$ is tracial, free cumulants are cyclically invariant, and
therefore the two terms above are equal. Thus,
\begin{equation}  \label{ec-r-diagonal sum}
    \kappa_{2m}(Y_g+Y_{g^{-1}},\ldots,Y_g+Y_{g^{-1}})
    =
    2\kappa_{2m}(Y_g,Y_{g^{-1}},\ldots,Y_g,Y_{g^{-1}}).
\end{equation}
    Our aim it then to calculate $\kappa_{2m}(Y_g,Y_{g^{-1}},\ldots,Y_g,Y_{g^{-1}})$, which we will do by considering the moment cumulant formula for $(Y_gY_{g^{-1}})^m$. To do this,  first notice that, given $X_g=\sum_{p\in G}E_{p,p}XE_{gp,gp}$ and $Y_g=U_g^\ast X_{g^{-1}}U_g^\ast$ from the Lemma \ref{lemma Y-UXU}, we have that
    \begin{equation*}
        Y_g=\sum_{p\in G}U_g^\ast E_{p,p}XE_{g^{-1}p,g^{-1}p}U_g^\ast=\sum_{p\in G}E_{g^{-1}p,g^{-1}p}XE_{p,p},
    \end{equation*}
    which implies that the product of $Y_gY_{g^{-1}}$ may be written as
\begin{align*}
    Y_gY_{g^{-1}}=&\sum_{p_1,p_2\in G}E_{p_1,p_1}XE_{gp_1,gp_1}E_{p_2,p_2}XE_{g^{-1}p_2,g^{-1}p_2}\\
    =&\sum_{p_1,p_2}E_{p_1,p_1}XE_{gp_1,p_2}XE_{g^{-1}p_2,g^{-1}p_2}\delta_{gp_1=p_2}\\
    =&\sum_{p\in G}E_{p,p}XE_{gp,gp}XE_{p,p}
\end{align*}
Now, taking the $m$ power evaluated in $\varphi$, we have
\begin{align*}
    \varphi\big((Y_gY_{g^{-1}})^{m}\big)=&\sum_{p_1,\ldots,p_{m}\in G}\varphi \Big(E_{p_1,p_1}XE_{gp_1,gp_1}XE_{p_1,p_1}E_{p_2,p_2}XE_{gp_2,gp_2}XE_{p_2,p_2}\cdots E_{p_{m},p_{m}}XE_{gp_{m},gp_{m}}XE_{p_{m},p_{m}}\Big)\\
    =&\sum_{p_1,\ldots,p_{m}\in G}\varphi \Big(XE_{gp_1,gp_1}XE_{p_1,p_1}E_{p_2,p_2}XE_{gp_2,gp_2}XE_{p_2,p_2}\cdots E_{p_{m},p_{m}}XE_{gp_{m},gp_{m}}XE_{p_{m},p_{m}}E_{p_1,p_1}\Big).
\end{align*}
Using that $E_{p_k,p_k}E_{p_{k+1},p_{k+1}}=E_{p_k,p_{k+1}}\delta_{p_k=p_{k+1}}$, we can observe that $\varphi\big((Y_gY_{g^{-1}})^{m}\big)$ is different from zero only if $p_{k+1}=p_k$ for every $k=1,2,\ldots,m$ (where $p_{m+1}:=p_1$). With this condition,
\begin{equation*}
    \varphi\big((Y_gY_{g^{-1}})^{m}\big)=\sum_{p\in G}\varphi \Big(XE_{gp,gp}XE_{p,p}XE_{gp,gp}XE_{p,p}\cdots XE_{gp,gp}XE_{p,p}\Big).
\end{equation*}
Since $X$ is free from $M_{|G|}(\mathbb{C})$, Theorem \ref{thm: as ind bs} implies that
\begin{equation*}
    \varphi\big((Y_gY_{g^{-1}})^{m}\big)=\sum_{p\in G}\sum_{\pi\in \mathcal{NC}(2m)}\kappa_{\pi}[X,\ldots,X]\varphi_{Kr(\pi)}[E_{gp,gp},E_{p,p},\ldots, E_{gp,gp},E_{p,p}],
\end{equation*}
and, exactly as in the proof of Proposition \ref{prop: Y_g}, each block of $Kr(\pi)$ has indices congruent modulo $2$ if and only if every block of $\pi$ has even size (Proposition 3.1 in \cite{arizmendi2012products}), so
\begin{equation}
    \label{ec-12}\varphi\big((Y_gY_{g^{-1}})^{m}\big)=\sum_{\substack{\pi\in \mathcal{NC}(2m)\\ \pi\text{ is even }}}|G|^{|\pi|}\,\kappa_{\pi}[|G|^{-1}X,\ldots,|G|^{-1}X].
\end{equation}
On the other hand, since $(Y_gY_{g^{-1}})^m$ is the ordered product $Y_gY_{g^{-1}}Y_gY_{g^{-1}}\cdots Y_gY_{g^{-1}}$ of $2m$ alternating factors, the moment--cumulant formula applied to this sequence gives
\begin{align}
    \varphi\big((Y_gY_{g^{-1}})^m\big)=&\sum_{\sigma\in\mathcal{NC}(2m)}\kappa_\sigma\big[Y_g,Y_{g^{-1}},\ldots,Y_g,Y_{g^{-1}}\big]\notag\\
=&\sum_{\substack{\sigma\in\mathcal{NC}(2m)\\ \pi\text{ is even}}}\kappa_\sigma\big[Y_g,Y_{g^{-1}},\ldots,Y_g,Y_{g^{-1}}\big],  \label{ec-11}
\end{align}
where the last equation follows since $\{Y_g,Y_{g^{-1}}\}$ is an $R$-diagonal pair, and then every block with non zero contribution must alternate between $Y_g$ and $Y_{g^{-1}}$ and thus must be have even size.
From the equations \eqref{ec-12} and \eqref{ec-11} one sees that 
\begin{equation*}
 \kappa_{2n}(Y_g,Y_{g^{-1}},\ldots,Y_g,Y_{g^{-1}})=|G|\kappa_{2n}(|G|^{-1}X,\ldots,|G|^{-1}X),
\end{equation*}
from where from \label{ec-r-diagonal sum}, one gets
\begin{equation*}
    \kappa_{2m}(Y_g+Y_{g^{-1}},\ldots,Y_g+Y_{g^{-1}})=2|G|\kappa_{2m}(|G|^{-1}X,\ldots,|G|^{-1}X)=2|G|^{1-2m}\kappa_{2m}(X,\ldots,X).
\end{equation*}
Finally, for $\nu:=\big(D_{1/|G|}\mu\big)^{\boxplus|G|}\boxplus\big(D_{-1/|G|}\mu\big)^{\boxplus|G|}$ we have $k_{2m+1}=0$ and
\[
\kappa_{2m}(\nu)=|G|^{1-{2m}}\big(1+(-1)^{2m}\big)\kappa_{2m}(X,\ldots,X)=2|G|^{1-2m}\kappa_{2m}(X,\dots,X),
\]
which coincides exactly with the cumulants of $Y_g+Y_g^{-1}$ just computed. This concludes the proof.
\end{proof}

\begin{cor}[Proposition 7.2 \cite{arizmendi2016asymptotic}]\label{main:corollary-1}
    The partial transpose $A^\Gamma_d=(id_d\otimes t)(A_d)\in M_d(\mathbb{C})\otimes M_{|G|}(\mathbb{C})$, where \(t:M_{|G|}(\mathbb{C})\to M_{|G|}(\mathbb{C})\) denotes the matrix transposition, has asymptotic eigenvalue distribution
\begin{equation*}
    \mu^\Gamma=\big(D_{\frac{1}{|G|}}\mu^{\boxplus |G|(|G|+1)/2}\big)\boxplus \big(D_{-\frac{1}{|G|}}\mu^{\boxplus |G|(|G|-1)/2}\big).
\end{equation*}
\end{cor}
\begin{proof}
    Combining the Propositions \ref{prop: Y_e}, \ref{prop: Y_g}, \ref{prop: Y_g+Y_g-1}, the main Theorem \ref{main-thm},
    we can conclude that
    \begin{align*}
        \mu^{\Gamma}=&\mu_{Y_e}\boxplus\big(\overset{s}{\underset{i=1}{\boxplus}} \mu_{Y_g}\big)\boxplus\left(\overset{r}{\underset{i=1}{\boxplus}}\mu_{Y_g+Y_{g^{-1}}}\right),
    \end{align*}
    where $s:=|S|$ and $r:=|R|=\frac{|G|-s-1}{2}$. Therefore 
    \begin{align*}
        \mu^{\Gamma}=&\big(D_{\frac{1}{|G|}}\mu\big)^{\boxplus |G|}\boxplus \left(\overset{s}{\underset{i=1}{\boxplus}} \big(D_{\frac{1}{|G|}}\mu\big)^{\boxplus\frac{|G|}{2}}\boxplus\big(D_{-\frac{1}{|G|}}\mu\big)^{\boxplus\frac{|G|}{2}}\right)\boxplus \left(\overset{r}{\underset{i=1}{\boxplus}} \big(D_{\frac{1}{|G|}}\mu\big)^{\boxplus|G|}\boxplus\big(D_{-\frac{1}{|G|}}\mu\big)^{\boxplus|G|} \right)\\
        =&\big(D_{\frac{1}{|G|}}\mu\big)^{\boxplus\big( |G|+\frac{|G|s}{2}+\frac{|G|(|G|-s-1)}{2}\big)}\boxplus \big(D_{-\frac{1}{|G|}}\mu\big)^{\boxplus\big( \frac{|G|s}{2}+\frac{|G|(|G|-s-1)}{2}\big)}\\
        =&\big(D_{\frac{1}{|G|}}\mu^{\boxplus |G|(|G|+1)/2}\big)\boxplus \big(D_{-\frac{1}{|G|}}\mu^{\boxplus |G|(|G|-1)/2}\big).
    \end{align*}
\end{proof}

\section{Applications to Wishart Matrices}\label{section-applicacions}

    Let $G$ be a fixed finite group, and consider the Wishart matrix defined in \eqref{ec-1}, with the convention $d_1=|G|$. Explicitly,
\begin{equation}\label{ec-3}
    W=\frac{1}{|G|d_2}
    \begin{bmatrix}
        G_1\\
        \rule{1cm}{0.5pt}\\
        \vdots\\
        \rule{1cm}{0.5pt}\\
        G_{|G|}
    \end{bmatrix}
    \begin{bmatrix}
        G_1^{*}&|&\cdots&|&G_{|G|}^{*}
    \end{bmatrix}
    =
    \frac{1}{|G|d_2}(G_iG_j^{*})_{i,j}
    \in M_{|G|d_2}(\mathbb C),
\end{equation}
where $G_1,\ldots,G_{|G|}$ are independent complex Gaussian random matrices of size $d_2\times p$. Its partial transpose is given by
\(
W^{\Gamma}
=(\mathrm{id}\otimes t)(W)
=\frac{1}{|G|d_2}(G_jG_i^{*})_{i,j}.
\)
Assuming that $|G|$ is fixed, and $d_2,p\to\infty$ in such a way that
\(
\frac{p}{d_2}\to c|G|
\), with $c>0$.
Let $\omega$ denote the limiting distribution of the Wishart matrix \eqref{ec-3}, and write
\[
|G|\omega=
\begin{bmatrix}
\omega_{1,1} & \omega_{1,2} & \cdots & \omega_{1,|G|}\\
\omega_{2,1} & \omega_{2,2} & \cdots & \omega_{2,|G|}\\
\vdots & \vdots & \ddots & \vdots\\
\omega_{|G|,1} & \omega_{|G|,2} & \cdots & \omega_{|G|,|G|}
\end{bmatrix}.
\]
Since the partial transpose exchanges the block indices, it follows that $|G|W^\Gamma$ has limit distribution
\(
|G|\omega^{t}
=
(\omega_{j,i})_{1\le i,j\le |G|}.
\) 
Moreover, $|G|\omega$ has the Marchenko--Pastur distribution of parameter $c$ and is free independent from $M_{|G|}(\mathbb{C})$. This implies that the left $G$-circulant decompositions from $|G|\omega^t$ 
\begin{equation}\label{ec-5}
    |G|\omega^t= A_e+A_{g_1}+\cdots +A_{g_{|G|-1}},
\end{equation}
is a free decomposition. Observe that, for every \(g\in G\),
\(
A_{g^{-1}}=A_g^*.
\)
Indeed, by definition,
\begin{equation*}
            A_g=\sum_{h\in G}A_{h,gh}\,|e_h\rangle\langle e_{gh}|.
        \end{equation*}
Therefore,
\begin{equation*}
            A_{g}^\ast=\Big(\sum_{h\in G}A_{h,gh}\,|e_h\rangle\langle e_{gh}|\Big)^\ast =\sum_{h\in G}\Big(A_{h,gh}\,|e_h\rangle\langle e_{gh}|\Big)^\ast =\sum_{h\in G}A_{gh,h}|e_{gh}\rangle\langle e_{h}|,
        \end{equation*}
Changing the variables \(k=gh\), we obtain $A_g^*= A_{g^{-1}}$.
\begin{cor} The family
    \(
    A_e,\,
    \{A_g\}_{g\in S},\,
    \{A_g,A_{g^{-1}}\}_{g\in R}
    \)
    is $\ast$-free.  
\end{cor}
According to the propositions in the last section, the free decomposition of \(|G|\omega^t\) in \eqref{ec-5} involve only three types of distributions that are very explicitly determined by the distribution of $\omega$ and the number of elements of the group $G$:
\begin{enumerate}
\item the distribution of the identity component \(A_e\),  that we call $\mu_e$;
\item a common distribution associated with nontrivial involutions, that we call $\mu_{inv}$;
\item a common distribution associated with inverse pairs \((g,g^{-1})\),  that we call $\mu_{pair}$.
\end{enumerate}
Therefore, the spectral limiting distribution of $|G|\omega^t$ is 
\begin{equation*}
    \mu_{|G|\omega^t}
=
\mu_e
\boxplus
\mu_{\mathrm{inv}}^{\boxplus s}
\boxplus
\mu_{\mathrm{pair}}^{\boxplus r},
\end{equation*}
where $s$ and $r$ are described in the last section.

\begin{Example}
    Continuing with Remark \ref{main-remark}, we consider the groups $G_1=S_3$ and $G_2=\mathbb{Z}_6$, and let $A \in M_{6\times 6}(\mathcal{A})$ be free over $M_{6\times 6}(\mathbb{C})$. The left-$G_1$ and $G_2$ circulant decompositions of $A^t$ are given by 
    \begin{equation*}
        A^t = A_0 + A_1 + A_2 + A_3 + (A_4 + A_5)
    \end{equation*}
    and
    \begin{equation*}
        A^t = B_0 + (B_1 + B_5) + (B_2 + B_4) + B_3,
    \end{equation*}
    respectively. By our previous results, we obtain the following identities in distribution:
    \begin{equation}
        \label{ec-8}
        A_0\stackrel{d}{=}B_0,\qquad A_1\stackrel{d}{=}A_2\stackrel{d}{=}A_3\stackrel{d}{=}B_3,\qquad\text{and}\qquad   A_4+A_5\stackrel{d}{=}B_1+B_5\stackrel{d}{=}B_2+B_4.
    \end{equation}
    To numerically illustrate these identities, we simulate a complex Wishart matrix block-structure. for this simulation, we consider $36$ blocks of size $n=400$. Consequently, the generated Wishart matrix has dimensions $2400 \times 2400$ with $400$ degrees of freedom. Figure \ref{fig-simulaciones} displays the empirical spectral distributions for three different types of configurations presented in \eqref{ec-8}.

\begin{figure}[h]
    \centering
    \includegraphics[width=1\linewidth]{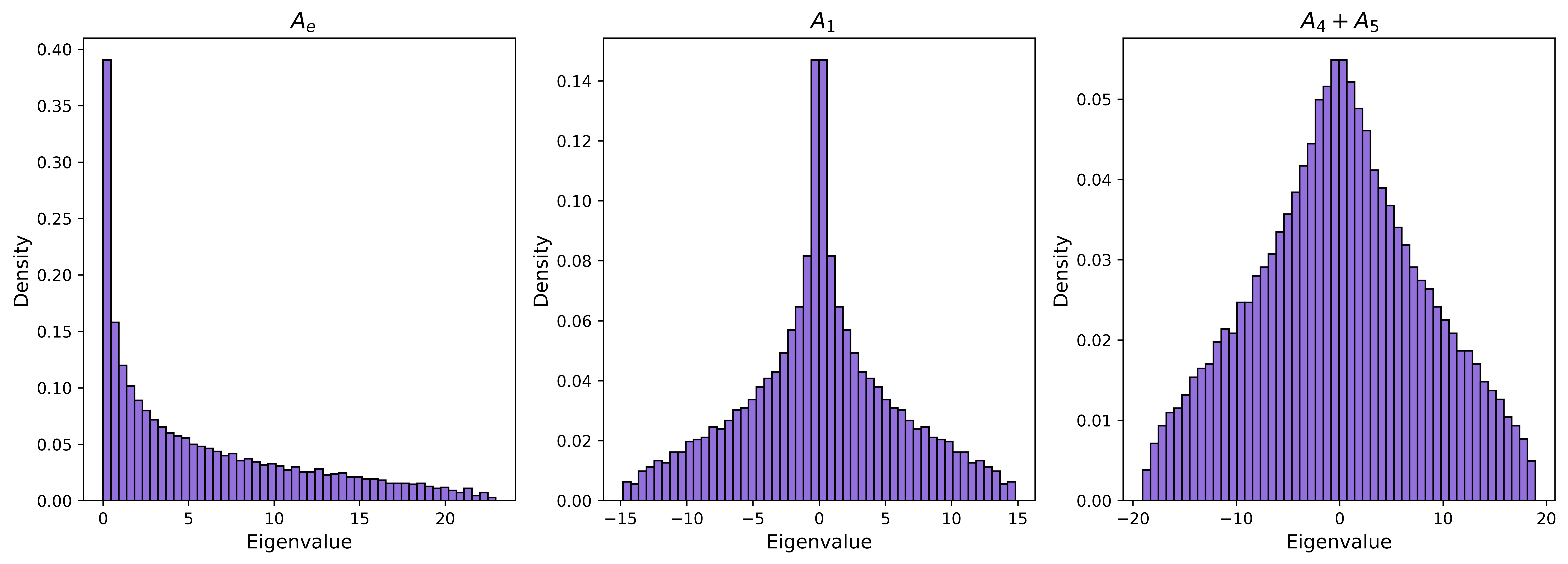}
    \caption{Empirical spectral distributions for the component $A_0$ (left), an involution $A_1$ (center), and the self-adjoint sum $A_4 + A_5$ (right) extracted from a $2400 \times 2400$ Wishart matrix.}
    \label{fig-simulaciones}
\end{figure}
    
\end{Example}
\newpage 
\bibliography{arXivVersion}

@article{speicher1994multiplicative,
  title={Multiplicative functions on the lattice of non-crossing partitions and free convolution},
  author={Speicher, Roland},
  journal={Mathematische Annalen},
  volume={298},
  number={1},
  pages={611--628},
  year={1994},
  publisher={Springer}
}

@book{nica2006lectures,
  title={Lectures on the combinatorics of free probability},
  author={Nica, Alexandru and Speicher, Roland},
  volume={13},
  year={2006},
  publisher={Cambridge University Press}
}

@article{diaconis1981generating,
  title={Generating a random permutation with random transpositions},
  author={Diaconis, Persi and Shahshahani, Mehrdad},
  journal={Zeitschrift f{\"u}r Wahrscheinlichkeitstheorie und verwandte Gebiete},
  volume={57},
  number={2},
  pages={159--179},
  year={1981},
  publisher={Springer}
}

@article{arizmendi2012products,
  author  = {Arizmendi, Octavio and Vargas, Carlos},
  title   = {Products of Free Random Variables and $k$-Divisible Non-Crossing Partitions},
  journal = {Electronic Communications in Probability},
  volume  = {17},
  number  = {18},
  pages   = {1--13},
  year    = {2012},
}

@article{aubrun2012partial,
  title={Partial transposition of random states and non-centered semicircular distributions},
  author={Aubrun, Guillaume},
  journal={Random Matrices: Theory and Applications},
  volume={1},
  number={02},
  pages={1250001},
  year={2012},
  publisher={World Scientific}
}

@article{peres1996separability,
  title={Separability criterion for density matrices},
  author={Peres, Asher},
  journal={Physical Review Letters},
  volume={77},
  number={8},
  pages={1413},
  year={1996},
  publisher={APS}
}

@article{banica2013asymptotic,
  title={Asymptotic eigenvalue distributions of block-transposed Wishart matrices},
  author={Banica, Teodor and Nechita, Ion},
  journal={Journal of Theoretical Probability},
  volume={26},
  number={3},
  pages={855--869},
  year={2013},
  publisher={Springer}
}

@article{banica2012block,
  title={Block-modified Wishart matrices and free Poisson laws},
  author={Banica, Teodor and Nechita, Ion},
  journal={Houston Journal of Mathematics},
  volume={41},
  number={1},
  pages={113--134},
  year={2015}
}

@article{jivulescu2014reduction,
  title={On the reduction criterion for random quantum states},
  author={Jivulescu, Maria Anastasia and Lupa, Nicolae and Nechita, Ion},
  journal={Journal of Mathematical Physics},
  volume={55},
  number={11},
  year={2014},
  publisher={AIP Publishing}
}

@article{arizmendi2016asymptotic,
  title={On the asymptotic distribution of block-modified random matrices},
  author={Arizmendi, Octavio and Nechita, Ion and Vargas, Carlos},
  journal={Journal of Mathematical Physics},
  volume={57},
  number={1},
  year={2016},
  publisher={AIP Publishing}
}

@inproceedings{voiculescu2006symmetries,
  title={Symmetries of some reduced free product C*-algebras},
  author={Voiculescu, Dan},
  booktitle={Operator Algebras and their Connections with Topology and Ergodic Theory: Proceedings of the OATE Conference held in Bu{\c{s}}teni, Romania, Aug. 29--Sept. 9, 1983},
  pages={556--588},
  year={2006},
  organization={Springer}
}

@article{voiculescu1995operations,
  title={Operations on certain non-commutative operator-valued random variables, in Recent Advances in Operator Algebras},
  author={Voiculescu, Dan},
  journal={Asterisque},
  volume={232},
  pages={243--275},
  year={1995}
}

@article{mingo2013real,
  title={Real second order freeness and Haar orthogonal matrices},
  author={Mingo, James A and Popa, Mihai},
  journal={Journal of Mathematical Physics},
  volume={54},
  number={5},
  year={2013},
  publisher={AIP Publishing}
}

@article{arizmendi2021cyclic,
  title={The cyclic group and the transpose of an $ R $-cyclic matrix},
  author={Arizmendi, Octavio and Mingo, James A},
  journal={Journal of Operator Theory},
  volume={85},
  number={1},
  pages={135--151},
  year={2021}
}

@article{voiculescu1991limit,
  title={Limit laws for random matrices and free products},
  author={Voiculescu, Dan},
  journal={Inventiones mathematicae},
  volume={104},
  number={1},
  pages={201--220},
  year={1991},
  publisher={Springer}
}

@article{mingo2022partial,
  title={On the partial transpose of a Haar unitary matrix},
  author={Mingo, James A and Popa, Mihai and Szpojankowski, Kamil},
  journal={Studia Mathematica},
  volume={266},
  number={3},
  year={2022}
}

@article{bolanos2023g,
  title={G-Circulant Quantum Markov Semigroups},
  author={Bola{\~n}os-Serv{\'\i}n, Jorge R and Quezada, Roberto and V{\'a}zquez-Becerra, Josu{\'e}},
  journal={Open Systems \& Information Dynamics},
  volume={30},
  number={01},
  pages={2350002},
  year={2023},
  publisher={World Scientific}
}
\bibliographystyle{alpha}
\end{document}